\documentclass{amsart}

\usepackage{amsmath}
\usepackage{amssymb}
\usepackage{amsthm}
\usepackage{thmtools}
\usepackage{mathtools}
\usepackage[mathscr]{eucal}
\usepackage{esint}
\usepackage{enumerate}
\usepackage[foot]{amsaddr}
\def\Xint#1{\mathchoice
{\XXint\displaystyle\textstyle{#1}}%
{\XXint\textstyle\scriptstyle{#1}}%
{\XXint\scriptstyle\scriptscriptstyle{#1}}%
{\XXint\scriptscriptstyle\scriptscriptstyle{#1}}%
\!\int}
\def\XXint#1#2#3{{\setbox0=\hbox{$#1{#2#3}{\int}$}
\vcenter{\hbox{$#2#3$}}\kern-.5\wd0}}

\def\dashint{\Xint-}

\usepackage{tikz}
\usetikzlibrary{cd}
\usepackage{xcolor}
\usepackage{slashed}
\usepackage{tensor}

\usepackage[colorlinks=true,linkcolor=blue,citecolor=blue]{hyperref}

\newtheorem{theorem}{Theorem}[section]
\newtheorem{corollary}[theorem]{Corollary}
\newtheorem{lemma}[theorem]{Lemma}
\newtheorem{remark}[theorem]{Remark}
\newtheorem{proposition}[theorem]{Proposition}

\newtheorem{definition}[theorem]{Definition}

\newtheorem{notation}[theorem]{Notation}

\usepackage[capitalise,noabbrev]{cleveref}

\newcommand{\R}{\mathbb{R}}
\newcommand{\Z}{\mathbb{Z}}
\newcommand{\N}{\mathbb{N}}

\newcommand{\U}{\mathbb{U}}
\newcommand{\T}{\mathbb{T}}
\newcommand{\F}{\mathcal{F}}
\newcommand{\Sp}{\mathbb{S}}
\newcommand{\A}{\mathcal A}
\newcommand{\cH}{\mathcal H}

\newcommand{\Ric}{\operatorname{Ric}}
\newcommand{\Scal}{\operatorname{Scal}}
\newcommand{\Scalminus}{\Scal^{-}}
\newcommand{\RCD}{\operatorname{RCD}}
\newcommand{\diver}{\operatorname{div}}
\newcommand{\Lip}{\operatorname{Lip}}
\newcommand{\loc}{\operatorname{loc}}
\newcommand{\vol}{\operatorname{Vol}}
\newcommand{\supp}{\operatorname{supp}}
\newcommand{\TestF}{\operatorname{TestF}}
\newcommand{\TestV}{\operatorname{TestV}}
\newcommand{\Hess}{\operatorname{Hess}}

\title[Torus and PMT stability with Ricci bounds]{Torus and Positive Mass stability for metrics with Ricci curvature lower bound}
\author{Edward Bryden$^1$}
\thanks{$^1$Funded by the FWO (grant 12F0223N)}
\address{$^1$University of Antwerp}
\author{Zhizhang Xie$^2$}
\thanks{$^2$partially supported by NSF 2247322}

\address{$^2$Texas A\&M University}
\email{etbryden@gmail.com, xie@tamu.edu}
\begin{document}
\begin{abstract}
    Consider a sequence of metrics $g_i$ on the torus whose
    members have uniform lower bounds on their first stable systoles
    and Ricci curvatures, and have an uniform upper bound on their
    diameters. If the $L^1$ norm of the negative part of the
    scalar curvatures vanishes along this sequence, then we show
    there is a subsequence of the metrics which converges in the
    measured-Gromov-Hausdorff topology to a flat metric on the torus.
    Something analogous holds for sequences of asymptotically flat spin Riemannian manifolds with a uniform lower bound on Ricci curvature and nonnegative scalar curvature: if the ADM masses of the distinguished ends tend to zero along the sequence, then the manifolds converge to Euclidean space in the pointed measured Gromov–Hausdorff sense.
\end{abstract}
\maketitle
\section{Introduction}
Understanding how curvature constrains geometry and topology is an enormous undertaking, and has been extremely fruitful.
Two well known and widely celebrated achievements in this direction are the theory of sectional curvature upper and lower bounds,
and the theory of Ricci curvature lower bounds.
Given the importance of understanding sectional curvature and Ricci curvature, one naturally wonders what constraints scalar curvature
might impose.
Since scalar curvature is just a function, we should expect that metrics with bounded scalar curvature should still exhibit a great deal
of flexibility, which is true.
For example, in \cite{Kazdan-Warner_prescribed_scalar_curvature} J. Kazdan and F. Warner show that for a 
compact manifold with dimension greater than  
or equal to 3, any function which is negative somewhere is the scalar curvature of some Riemannian metric on the manifold.
We therefore see that scalar curvature upper bounds cannot tell us much in general.
In fact, J. Lohkamp later showed in \cite{metrics_of_negative_Ricci_curvature} that any manifold admits complete metrics
with negative Ricci curvature. 
The only hope is that scalar curvature lower bounds might give some control over metrics.
There are several well known results in this direction: the Positive Mass Theorem, see \cite{SchoenYauPMT,Witten-PMT},
Llarull's theorem \cite{Llarull}, and the Torus Rigidity Theorem, see \cite{schoen_yau_79, GL1}.
Let us recall the statement of the Torus Rigidity Theorem, namely a Riemannian metric on a torus is flat if and only if
it has nonnegative scalar curvature.
This has been proven using minimal surfaces \cite{schoen_yau_79}, using spinors \cite{GL1},
and, in the three dimensional case, using nontrivial harmonic maps into $\Sp^{1}$.
Let us review some aspects of the latter two approaches, starting anachronistically with the harmonic map approach.

In \cite{Stern} D. Stern demonstrated the following 
inequality:
\begin{theorem}[{\cite{Stern}}]
    Let $(M^3,g)$ be a closed, oriented 3-manifold, and let
    $u:M\rightarrow\Sp^1$ be a nontrivial harmonic map.
    Denote by $\Sigma_{\theta}=u^{-1}\lbrace\theta\rbrace$ the level sets of $u$, and $\chi(\Sigma_{\theta})$ their
    Euler characteristic when they are smooth. 
    Then, we have
    \begin{equation}
        2\pi\int_{\Sp^1}\chi(\Sigma_{\theta})d\theta\geq
        \frac12\int_{M}\frac{|\nabla^2u|^2}{|du|}+R_g|du|
        dV_g.
    \end{equation}
\end{theorem}
This has important consequences for 3-dimensional Riemannian manifolds in general, and for
metrics on the torus in particular.
Specifically, we are led to the following corollary.
\begin{corollary}[{\cite{Stern}}]\label{cor:3-dim_scal_formula}
    Let $g$ be a Riemannian metric on $\T^3$, let $\Scalminus_g$ be
    the negative part of the scalar curvature, and let $u:(\T^3,g)\rightarrow\Sp^1$
    be a nontrivial harmonic map. Then, we have
    \begin{equation}
        \int_{\T^3}|du|\Scalminus_gdV_g\geq\int_{\T^3}\frac{|\nabla^2u|^2}{|du|}
        dV_g.
    \end{equation}
\end{corollary}
The above inequality leads to a quick proof of the Torus Rigidity Theorem in 3 dimensions.
Before stating and establishing this result, let us recall the following notation.
\begin{notation}
    We will denote by $H^1(\T^n;\Z)_{\R}$ and $H_1(\T^n;\Z)_{\R}$ the lattice of integer valued
    cohomology classes sitting inside $H^{1}(\T^n,\R)$, and the lattice of integral homology classes
    sitting inside $H_1(\T^n;\R)$, respectively.
\end{notation}
\begin{theorem}[\cite{Stern}]
    Let $g$ be a Riemannian metric on $\T^3$ and suppose that
    $R_g\geq 0$, in other words $\Scalminus_g=0$. Then $g$ is flat.
\end{theorem}
\begin{proof}
    Take an harmonic representative of any nonzero element of $H^1(\T^n;\Z)_{\R}$.
    This can be thought of as $du$ where $u:(\T^3,g)\rightarrow\Sp^1$ is a
    nontrivial harmonic map.
    The inequality above along with the assumption that $\Scalminus_g=0$
    forces the Hessian of $u$ to vanish identically.
    Therefore $u$ splits $(\T^3,g)$ into $(\Sp^1\times\T^2,d\theta^2+\widetilde g)$.
    Since scalar curvature is additive over product metrics, we see that $g$
    must be flat by the Gauss-Bonnet theorem.
\end{proof}
The explicit relationship between scalar curvature and the Hessian of non-trivial harmonic maps
given by Stern's inequality invites one to wonder
whether a stability result holds as well: if scalar curvature is nearly nonnegative,
then is $g$ nearly flat?
This is also sometimes referred to as almost rigidity.
One major obstacle to such a stability result is that
metrics with scalar curvature lower bounds which are not sharp, 
for example $\mathrm{Scal}\geq -\varepsilon$ instead of $\mathrm{Scal}\geq0$ on a torus,
can still exhibit very surprising phenomena. 
This can be seen in the constructions of M-C. Lee, A. Naber, and R. Neumayer in \cite{d_p-convergence_and_epsilon_regularity}, of
D. Kazaras and K. Xu in \cite{kazaras2023drawstringsflexibilitygerochconjecture}, and of P. Sweeney in \cite{examples_scalar_sphere_stability}. 
They each provide a way to perturb a known metric in a way which dramatically alters 
its distance structure, but only decreases the scalar curvature by an arbitrarily small amount.

One way out of this conundrum is to find a notion of convergence which ignores, or cuts out, such pathological behavior.
This is the approach taken in \cite{d_p-convergence_and_epsilon_regularity}, and the approach C. Dong and A. Song take
in their proof of the stability of Euclidean 3 space for the Positive Mass Theorem \cite{Dong-Song}.
Another way is to add an extra condition which serves to limit the type of instabilities which can arise.
The first named author and L. Chen took aspects of both approaches, along with Stern's inequality, to establish
stability for a class of metrics on $\T^{3}$ in \cite{Bryden-Chen}.

For higher dimensional tori,
S. Honda, C. Ketterer, I. Mondello, R. Perales, and C. Rigoni
in \cite{Honda-Ketterer-Christian-Mondello-Perales-Rigoni}
prove an equivalence for spaces with Ricci curvature bounded
below between the existence of a special harmonic map with small
Hessian in the $L^2$ sense and being measure-Gromov-Hausdorff close to
a flat metric on $\T^n$.
For 3-tori with small negative scalar curvature, the existence of a good harmonic map with small Hessian
in the $L^2$ norm is given by Corollary \ref{cor:3-dim_scal_formula},
and so they are able to obtain the stability result in this case.

Our first main result is the following.
\begin{theorem}\label{thm:main_result-torus_stability}
  For $n\in\mathbb{N}$, $\sigma>0$, $D<\infty$ and $K<\infty$,
  let $\mathcal{F}(n,\sigma, D,K)=\mathcal{F}$ denote the family of Riemannian metrics on $\mathbb{T}^{n}$ such that
  \begin{enumerate}[(a)]
    \item the  stable $1$-systole\footnote{See Definition \ref{def:stable_norm_of_an_homology_class} for the precise definition. } of $g$ satisfies $\mathrm{stabsys}_{1}(g)\geq\sigma$;\label{assumption:stablesystolic_lower_bound}
    \item the diameter of $g$ satisfies $\mathrm{diam}(g)\leq D$;\label{assumption:diameter_upper_bound}
    \item the Ricci curvature of $g$ satisfies $\mathrm{Ric}_{g}\geq-K $.\label{assumption:Ricci_lower_bound}
  \end{enumerate}
  Then, for any sequence $g_{i}\in \mathcal{F}$ such that $\|\Scalminus_{g_i}\|_{L^{1}}\rightarrow0$, there is a subsequence
  $g_{i_{j}}$ and a flat metric $g_{\infty}$ on $\mathbb T^n$ such that
  \begin{equation}
    \lim_{j\rightarrow \infty }d_{\mathrm{mGH}}\left(
      \left(\mathbb{T}^{n},d_{g_{i_{j}}},\mathrm{vol}_{g_{i_{j}}}\right),
      \left(\mathbb{T}^{n},d_{g_{\infty }},\mathrm{vol}_{g_{\infty }}\right)
    \right)=0,
  \end{equation}
where $d_{\mathrm{mGH}}$ stands for the measure-Gromov-Hausdorff distance. 
\end{theorem}
We would like to thank Shouhei Honda for pointing out the following corollary and
its proof; see Definition~\ref{def:stable_norm_of_an_homology_class} for the
necessary background.
\begin{corollary}
    Let $D,K>0$. Consider a sequence of Riemannian metrics $g_{i}$ on $\T^n$ such that
    $\displaystyle\sup_{i}\textrm{diam}(g_i)\leq D$, $\Ric_{g_i}\geq -K$ for all $i$, 
    and $\displaystyle\lim_{i\to\infty}\|\Scalminus_{g_i}\|_{L^1(g_i)}=0$, then the following are equivalent:
\end{corollary}
\begin{enumerate}
        \item there is a subsequence that
    measure Gromov-Hausdorff converges to a flat $n$-dimensional torus $(\T^n,g_{\infty})$; 
    \item there is a subsequence whose stable $1$-systoles are uniformly bounded below by a positive number. 
    \end{enumerate}
\begin{proof}
    From \cite[Theorem 1.1]{Honda2023} we know that for any $\varepsilon\in(0,1)$
    there exists a $\delta\in(0,1)$ such that if a compact $\textrm{RCD}(K,N)$
    space, say $(X,d,m)$, is $\delta$-Gromov-Hausdorff close to an $n$-dimensional
    Riemannian manifold $(M^{n},g)$, which as a metric-measure space is the triple
    $(M^{n},d_{g},\textrm{vol}_{g})$, then there exists a homeomorphism $F$
    from $X$ to $M^{n}$ such that for all $x,y\in X$ we have
    \[
    (1-\varepsilon)d(x,y)^{1+\varepsilon}\leq d_{g}(F(x),F(y))\leq(1+\varepsilon)
    d(x,y).
    \]
    Now consider a sequence of Riemannian manifolds $(M_{i},g_{i})$ with 
    a uniform lower bound on their Ricci curvatures.
    Furthermore, suppose that $(M_i,g_{i})$ measure Gromov-Hausdorff converges
    to a Riemannian manifold $(N,g_{\infty})$.
    Fix any $\varepsilon=(0,1)$; by the above there exists an $i_{0}$ such that
    for all $i\geq i_0$ there is an homeomorphism $F_{i}:M_{i}\to N$ such that
    for all $x,y\in M_{i}$ we have $d_{g_{\infty}}(F(x),F(y)) \leq(1+\varepsilon)d_{g_i}(x,y)$.
    Suppose that $\gamma_{i}$ represents a nontrivial element of $H_1(M_{i},\Z)_{\R}$.
    Since $F_{i}$ is an homeomorphism the element $F(\gamma_{i})$ is
    of course a nontrivial element of $H_{1}(N,\Z)_{\R}$.
    Furthermore, since $F_{i}$ is $(1+\varepsilon)$-Lipschitz, we see that
    $\textrm{vol}(F(\gamma_{i}))\leq(1+\varepsilon)\textrm{vol}(\gamma_{i})$.
    This shows that 
    $\textrm{stabsys}_{1}(N,g_{\infty})\leq\liminf_{i\to\infty}\textrm{stabsys}_{1}(M_{i},g_{i})$,
    and so one direction of the equivalence is taken care of.
    The other direction is an application of Theorem~\ref{thm:main_result-torus_stability}.
\end{proof}

The approach taken in the present paper  will be to study harmonic maps in conjunction with harmonic spinors.
Let us recall an inequality for harmonic spinors arising from Lichnerowicz's formula.
\begin{lemma}
    Let $g$ be a Riemannian metric on $\T^n$, $S$ the associated spinor bundle, and $E_0$
    an Hermitian vector bundle on the standard round sphere $\Sp^n$ such that
    $\langle\mathrm{ch}_n(E_0),[\Sp^n]\rangle\neq0$.
    Consider a smooth map with nonzero degree, say
    $\phi:(\T^n,g)\rightarrow(\Sp^n,g_{\mathrm{round}})$, with
    $\mathrm{Lip}(\phi)\leq\varepsilon$, and set $E=\phi^*E_0$.
    Then, there is a nontrivial harmonic spinor $s$ of the twisted bundle $S\otimes E$ (that is, $D^E(s) =0$ for the twisted Dirac operator  $D^E$  associated with the bundle $S\otimes E$) and it satisfies the inequality
    \begin{equation}
        \int_{\T^n}|\nabla s|^2dV_g\leq
        \int_{\T^n}
        \left(\tfrac{1}{4}\Scalminus_g+C_{E_{0}}\varepsilon^2\right)|s|^2dV_g.
    \end{equation}
where $C_{E_0}$ is a positive constant independent of $g$ and $\phi$.
\end{lemma}
\begin{remark}
  Stern's formula can be viewed as a version of Lichnerowiz's formula which applies to harmonic maps instead of 
  harmonic spinors.
\end{remark}
The formula above is closely related to the starting point for the proof of the Torus Rigidity Theorem using spinors,
see \cite[Chapter IV Theorem 5.5]{lawson2016spin}.
We can now give an overview of the proof of Theorem \ref{thm:main_result-torus_stability}.
Let $\pi^{k}:\T^{n}\rightarrow \T^{n}$ be the $k^{n}$-fold covering map, and let $g^{k}=(\pi^{k})^{*}g$ be its pullback
metric.
In Section \ref{sec:uniform_enlargeability} we show that there is a constant $C_{\mathcal{F}}$ so that for any metric
in the family $\mathcal{F}$ one can find harmonic spinors $s^{k}$ on $(\T^{n},g^{k})$ such that
\begin{equation}
  \int_{\T^{n}}|\nabla s^{k}|^{2}dV_{g^{k}}\leq 
  \int_{\T^{n}}\left(\frac{1}{4}\Scalminus_{g}+C_{\mathcal{F}}k^{-2}\right)\lvert s^{k} \rvert^{2}dV_{g^{k}}.
\end{equation}
In Section \ref{sec:scal_bounds_harmonic_1-forms} we let $\omega$ be an arbitrary harmonic 1-form on $(\T^{n},g)$ and
multiply the harmonic spinor $s^{k}$ by $\omega^{k}=(\pi^{k})^{*}\omega$.
This produces a new spinor with the property that $\|\nabla(\omega^{k}\cdot s^{k})\|_{L^{2}(g)}$
is controlled in terms of $\|\Scalminus_{g}\|_{L^{1}(g)}$ and $k^{-2}$, although it generally will
not be harmonic. The Clifford multiplication can be partially reversed, and so one obtains control of $|\nabla \omega^{k}|$ in terms of
$|s^k|$, $|\nabla s^{k}|$, and $|\nabla (\omega^{k}\cdot s^{k})|$. Since $\nabla \omega^{k}$ is invariant under covering transformations,
averaging out the above relationship leads to information on $\nabla \omega$ over $(\T^{n},g)$.
Under the assumption of a Ricci curvature lower bound, this turns out to be enough to bound $\|\nabla \omega\|_{L^{2}(g)}$ in terms of 
$\|\omega\|_{L^{\infty }}$ and $\|\Scalminus_{g}\|_{L^{1}(g)}$.
In Section \ref{sec:stability} we use the above results to construct a degree 1 harmonic map 
$\U:(\T^{n},g)\rightarrow(\T^{n},g_{0})$ with $\|\nabla d\U\|_{L^{2}(g)}$ controlled in terms of $\|\Scalminus_g\|_{L^{1}(g)}$,
where $(\mathrm{T}^{n},g_{0})$ is the torus with the standard product metric $g_{0}$.
Then, we explicitly show that this map satisfies the conditions of
\cite[Theorem 1.4]{Honda-Ketterer-Christian-Mondello-Perales-Rigoni},
which completes the proof.

Our second main result concerns the stability of the Positive Mass Theorem.  In dimension three, H. Bray, D. Kazaras, M. Khuri, and D. Stern obtained
a mass inequality involving the Hessians of asymptotically linear harmonic functions
\cite{Bray-Kazaras-Khuri-Stern}. D. Kazaras, M. Khuri, and D. Lee used this inequality to prove
pointed Gromov--Hausdorff stability under a uniform Ricci curvature lower bound
\cite{Kazaras:2021rfv}. Later, C. Dong and A. Song proved pointed measured
Gromov--Hausdorff stability without a Ricci curvature lower bound after excising subsets whose
boundary areas tend to zero, and using the induced length metric on what remains \cite{Dong-Song}.
In higher dimensions J. Klemmensen \cite{PMT-stab_Kaehler} proved the stability of the 
Positive Mass Theorem for K{\"a}hler manifolds of all dimensions; there is both
a version with and without a uniform Ricci curvature lower bound.

In this paper, we prove the stability of
the Positive Mass Theorem for spin Riemannian manifolds in all
dimensions, under the assumption of  a uniform lower bound on
Ricci curvatures. See \cref{def:b-tau-asymflat} for the precise meaning of undefined terms in the theorem below. 
\begin{theorem}\label{thm:pmt-stability-intro}
	Fix an integer $n\ge3$, constants $b>0$, $\tau>(n-2)/2$ and $K\ge0$,
	and a point
	\[
	\mathbf p\in\R^n\setminus\overline B^{\mathbb E}_1(0).
	\]
	Let $(M_i,g_i)$ be a sequence of $n$ dimensional, connected, complete, asymptotically flat
  spin manifolds without boundary, each having finitely many asymptotically flat ends
  $\mathcal E_{i,0},\ldots,\mathcal E_{i,N_i}$.  Suppose the distinguished end
  $\mathcal E_{i,0}$ has a chart $\Phi_i:=\Phi_{i,0}$ which is
  $(b,\tau)$-asymptotically flat.   Assume
	\begin{equation}\label{eq:curvature-assumptions-intro}
		\Scal_{g_i}\ge0,
		\qquad
		\Ric_{g_i}\ge-K g_i,
	\end{equation}
	and set $p_i:=\Phi_i^{-1}(\mathbf p)$.  If the ADM mass of the distinguished end $\mathcal E_{i, 0}$  converges to zero,
	then 	$(M_i,d_{g_i},p_i)$ converges to the Euclidean space $(\R^n,d_{\mathbb E},0)$ in the pointed Gromov-Hausdorff sense.
\end{theorem}

The paper is organized as follows. Section~\ref{sec:background} reviews stable systoles, lattices,
and some geometric and analytic tools associated with Ricci curvature lower bounds. Section~\ref{sec:uniform_enlargeability}
establishes uniform enlargeability for the torus metrics under consideration. Section~\ref{sec:scal_bounds_harmonic_1-forms}
uses twisted harmonic spinors to derive estimates for harmonic $1$-forms, and Section~\ref{sec:stability}
uses these estimates to prove the torus stability theorem. Section~\ref{sec:pmt-stability}
adapts $\RCD$ limit space techniques to the spin setting to proves the stability theorem for the Positive
Mass Theorem for spin Riemannian manifolds in all
dimensions, under the assumption of  a uniform lower bound on
Ricci curvatures.. 

We would like to thank Shouhei Honda and Antoine Song for helpful discussions. We also thank the Simons Center for Geometry and Physics for its hospitality, where this work was initiated.

\section{Background}\label{sec:background}
Since we will encounter many constants whose precise values are not important, but whose dependence on certain
parameters is, we will follow the convention indicated below.
\begin{notation}
  We will denote by $C\left(X_{1},\dots,X_{m}\right)$ a constant which depends only on the parameters $X_{1},\dots,X_{m}$.
  The precise value of such a constant may change from line to line.
  Furthermore, if two or more constants of this form are combined, say $C(X_{1},X_{2})$ and $C(X_{3},X_{4})$,
  then we will denote the new combined constant by $C\left(X_{1},\dots,X_{4}\right)$.
  We may abuse notation slightly, and occasionally write something of the form $C\left(n,KD^{2}\right)$.
  Of course, we have three parameters here: $n,K$, and $D$. So we should write $C(n,K,D)$, however the point is
  that $C(n,KD^{2})$ only depends on the value of $n$ and $KD^{2}$, something which $C(n,K,D)$ does not indicate.
  Finally, if a family of objects, say $\mathcal{F}$, is defined in terms of parameters $X_{1},\dots,X_{m}$,
  then we also write $C_{\mathcal{F}}$ in place of $C\left(X_{1},\dots,X_{m}\right)$ to indicate constants involved in 
  estimating quantities associated to the members of $\mathcal{F}$.
\end{notation}

Let us review some results about systoles, lattices, and spaces with Ricci curvature lower bounds
and diameter upper bounds that are key to the proof of Theorem \ref{thm:main_result-torus_stability}.
To start, we review the definition of the stable norm of a real homology class.

\begin{definition}[{See \cite[Def 12.1.4]{katz2007systolic} and \cite[\S 4.C]{Gromov-metric_structures}}]
\label{def:stable_norm_of_an_homology_class}
  Let $(M,g)$ be a Riemannian manifold, let $\triangle^{k}$ denote the standard $k$-simplex in $\R^{k}$, 
  and let $c$ be a real $k$-dimensional Lipschitz $k$-chain, that is there are real numbers $r_{i}$
  and Lipschitz maps $\sigma_{i}:\triangle^{k}\rightarrow M$ and $c$ is the sum
  \[
    \sum_{i}r_{i}\sigma_{i}.
  \]

  We define the volume of $c$ as follows:
  \begin{equation}
    \mathrm{vol}_{k}(c)=\sum_{i}|r_{i}|\mathrm{vol}_{\sigma_{i}^{*}g}\left(\triangle^{k}\right).
  \end{equation}
  For a $k$-cycle $c$, we denote its homology class by $[c]$.
  Then, the \textit{stable mass norm} of a real homology class $\alpha\in H_{k}(M;\R)$ is given by
  \begin{equation}
    \|\alpha\|_{\mathrm{M}}:=\inf\bigl\{\mathrm{vol}_{k}(c): c \textup{ is a real Lipschitz $k$-cycle  and } [c]=\alpha\bigr\}.
  \end{equation}
  Finally, the \textit{stable k-systole}, denoted $\mathrm{stabsys}_{k}(M,g)$, is defined to be
  \begin{equation}
    \mathrm{stabsys}_{k}(M,g)=\min\bigl\{\|\alpha\|_{\mathrm{M}}:\alpha \neq 0\in H_{k}(M;\Z)_{\R}\subset H_{k}(M;\R)\bigr\},
  \end{equation}
  where $H_{k}(M;\Z)_{\R}$ denotes the lattice of integer $k$-homology classes sitting inside of
  the real $k$-homology classes.
\end{definition}
For the present work, we are only interested in $\mathrm{stabsys}_{1}(M,g)$.
The importance of this quantity to us lies in its relationship to $H^{1}(M;\R)$ given by
the Kronecker pairing
\[
  \langle\cdot,\cdot\rangle:H_{1}(M;\R)\times H^{1}(M;\R)\rightarrow \R.
\]
In particular, this pairing induces a norm on $H^{1}(M;\R)$ dual to the stable mass norm. This dual norm is 
the \textit{comass norm}, whose definition we recall now.
\begin{definition}
  For a $p$-form $\omega\in\Omega^{p}(M)$, define its pointwise comass by
  \begin{equation}
    |\omega|^{*}(x)
    :=\sup\left\{|\omega_{x}(e_{1},\dots,e_{p})|:
      \begin{array}{c}
        (e_{1},\dots,e_{p})\text{ is an orthonormal $p$-frame}\\
        \text{in }T_{x}M
      \end{array}
    \right\},
  \end{equation}
  and set $|\omega|^{*}:=\sup_{x\in M}|\omega|^{*}(x)$.
  The \emph{comass norm} of a class $a\in H^{p}(M;\R)$ is
  \begin{equation}
    \|a\|^{*}_{\mathrm{M}}
    :=\inf\bigl\{|\omega|^{*}: \omega\in\Omega^{p}(M),\ d\omega=0,\ [\omega]=a\bigr\}.
  \end{equation}
\end{definition}
Observe that when $p=1$ the comass norm and the $L^{\infty }$ norm agree, a fact we will use without further mention.
In general, we have the following:
\begin{proposition}[{See \cite[\S 1.8: Mass and Comass]{GMT_Fed}}]
  The comass norm is equivalent to the sup norm. That is, for each dimension $n$ and $1\leq p\leq n$ 
  for all $\omega\in \Omega^{p}(M)$ we have
  \begin{equation}
    \binom{n}{p}^{-\frac{1}{2}}\|\omega\|_{L^{\infty }}\leq \lvert \omega \rvert^{*}\leq \|\omega\|_{L^{\infty }}.
  \end{equation}
\end{proposition}
The elements of $H^{1}(M;\Z)_{\R}\subset H^{1}(M;\R)$ are of particular use because
they can be represented by smooth maps $u:M\rightarrow \Sp^{1}$.
A natural idea is to combine multiple maps to the circle to get a map to the torus.
In order to get a handle on these maps, it is important to understand how their behavior
is influenced by the geometry of $(M,g)$.
We will show in Section \ref{sec:uniform_enlargeability} that the stable 1-systole 
will control the Lipschitz norms of these maps.
The key is that the lattice dual to $H_{1}(M;\Z)_{\R}$
is precisely $H^{1}(M;\Z)_{\R}$:
\[
  \left\{a\in H^{1}(M;\R):\langle\sigma,a\rangle\in \Z, ~ \forall \sigma\in H_{1}(M;\Z)_{\R}\right\}=H^{1}(M;\Z)_{\R}.
\]
To make all of this work, we require a few results on how lattices interact with arbitrary norms, which can be found in
\cite{cassels2012introduction}.
We begin by recalling the definition of the successive minima of a lattice.
\begin{definition}[{See \cite[Ch. VIII.1]{cassels2012introduction}}]\label{def:successive_minima}
  Let $\Lambda\subset \R^{n}$ be an $n$-dimensional lattice, and let $\|\cdot\|$ be a norm on
  $\R^{n}$. We define the $i^{th}$ successive minima of $\Lambda$ with respect to $\|\cdot\|$,
  denoted $\lambda_{i}$, as follows:
  \begin{equation}
    \lambda_{i}
    :=\inf\left\{\lambda>0:
      \begin{array}{c}
        \text{there exist linearly independent }a_{1},\dots,a_{i}\in\Lambda\\
        \text{such that }\|a_{s}\|\leq\lambda\text{ for }s=1,\dots,i
      \end{array}
    \right\}.
  \end{equation}
\end{definition}
\begin{remark}
  In \cite[Ch. VIII.1]{cassels2012introduction} successive minima are defined with respect to  distance functions, a class more
  general than norms.
\end{remark}
To get a quick sense for why successive minima are important,
consider the lattice $H_{1}(M;\Z)_{\R}$ with the mass norm.
Following the definitions, we see that $\lambda_{1}$ is the stable 1-systole in this case.
We now state a few results concerning successive minima, lattices, and norms.
\begin{theorem}[{\cite[Theorem VI, Ch. VIII \S 5]{cassels2012introduction}}]\label{thm:bounds_dual_minima}
  Let $\lambda_{1},\dots,\lambda_{n}$ be the successive minima of a lattice $\Lambda$ with respect to a symmetric convex 
  distance-function $F$ (in particular a norm) and let $\lambda_{1}^{*},\dots,\lambda_{n}^{*}$ be the successive minima
  of the polar lattice (referred to here as the dual lattice) with respect to the distance-function $F^{*}$ polar to $F$
  (in particular the dual norm).
  Then
  \begin{equation}
    1\leq \lambda_{j}\lambda^{*}_{n+1-j}\leq n!
  \end{equation}
\end{theorem}
\begin{theorem}[{\cite[Lemma 1 Ch. VII, \S 1.2]{cassels2012introduction}}]\label{thm:vectors_realizing_minima}
  Let $\|\cdot\|$ be a norm on $\R^{n}$, let $\Lambda$ be a lattice in $\R^{n}$, and let $\lambda_{i}$
  be the successive minima of $\Lambda$ with respect to $\|\cdot\|$.
  Then, there exist a collection of $n$ linearly independent points $\{a_{i}\}_{1}^{n}\subset \Lambda$
  such that 
  \begin{equation}
    \|a_{i}\|=\lambda_{i}.
  \end{equation}
  Furthermore, if $a\in \Lambda$ and $\|a\|<\lambda_{i}$, then $a$ is linearly dependent on $\{a_{l}\}_{l=1}^{i}$.
\end{theorem}
\begin{lemma}[{\cite[Lemma 8, Ch. V, \S 4]{cassels2012introduction}}]\label{lem:basis_from_sub_lattice_basis}
  Let $\|\cdot\|$ be a norm on $\R^{n}$ and $\{a_{i}\}_{i=1}^{n}$ be $n$ linearly independent points of a
  lattice $\Lambda$.
  Then, there exists a basis $\{b_{i}\}_{i=1}^{n}$ of $\Lambda$ for which
  \begin{equation}
    \|b_{j}\|\leq\max\left\{\|a_{j}\|,\frac{1}{2}\sum_{i=1}^{j}\|a_{i}\|\right\}.
  \end{equation}
\end{lemma}
\begin{corollary}\label{cor:existence_of_bounded_basis}
  Let $\|\cdot\|$ be a norm on $\R^{n}$, let $\Lambda$ be a lattice in $\R^{n}$, and let $\lambda_{i}$
  be the successive minima of $\Lambda$ with respect to $\|\cdot\|$.
  Then, there is a basis $\{b_{j}\}_{j=1}^{n}$ of $\Lambda$ such that
  \begin{equation}
    \|b_{j}\|\leq j\lambda_{j}.
  \end{equation}
\end{corollary}

Let us now move on to recalling some of the control one obtains on metrics with a lower bound on their
Ricci curvature and an upper bound on their diameter.


\begin{theorem}[{\cite[Theore 7.1.13]{petersen2016riemannian}}]\label{thm:Poincare_lower_bound-Ricci_lower_bound}
  Let $K\geq0$ and $(M,g)$ be a Riemannian manifold satisfying $\Ric_{g}\geq -K$. Denote by $D$ the diameter of $(M,g)$.
  Then there is a constant $C(n,KD^{2})$ such that for all $R\leq D$ and $\nu\in \left[1,\tfrac{n}{n-1}\right]$, the following holds.
  Set  $|B(x, R)|:=\mathrm{vol}_{g}(B(x, R))$, and define
  \[
    u_{B(x, R)}:=\frac{1}{|B(x, R)|}\int_{B(x, R)}u\,dV_{g}=\dashint_{B(x,R)}u\,dV_{g}.
  \]
  Then
  \begin{equation}
    \left(\dashint_{B(x,R)}|u-u_{B(x,R)}|^{\nu}dV_{g}\right)^{\frac{1}{\nu}}\leq
    C(n,KD^{2})R\,\dashint_{B(x,R)}|du|dV_{g}.
  \end{equation}
\end{theorem}
\begin{remark}
  In what follows, it will be convenient to define two related norms.
  First, we set
  \begin{equation}
    \|f\|_{\overline{L}^{p}(M,g)}=\left(\dashint_{M}|f|^{p}\right)^{\frac{1}{p}},
  \end{equation}
  and $f_{M}=\tfrac{1}{\mathrm{vol}_{g}(M)}\int_{M}f=\dashint_{M}f$.
  Second, we set
  \[
    \|f\|_{L^p(M,g)}=\left(\int_{M}|f|^{p}\right)^{\frac{1}{p}}.
  \]
\end{remark}
\begin{proposition}[{\cite[Proposition 7.1.17]{petersen2016riemannian}}]\label{prop:global_Poincare_from_Ricci_diameter_and_volume}
  Let $(M,g)$ be a Riemannian manifold such that for $s>1$ all smooth functions satisfy the inequality
  \begin{equation}
    \|u-u_{M}\|_{\overline{L}^{\frac{s}{s-1}}(M,g)}\leq S\|du\|_{\overline{L}^{1}(M,g)},
  \end{equation}
  then for $1\leq p<s$ we have
  \begin{equation}
    \|u\|_{\overline{L}^{\frac{sp}{s-p}}(M,g)}\leq\frac{p(s-1)}{s-p}S\|du\|_{\overline{L}^{p}(M,g)}+\|u\|_{\overline{L}^{p}(M,g)}.
  \end{equation}
\end{proposition}
It follows from Theorem~\ref{thm:Poincare_lower_bound-Ricci_lower_bound} that a family of Riemannian
manifolds with a uniform upper bound on their diameters and a uniform lower bound on their Ricci curvatures
will satisfy the hypothesis of Proposition~\ref{prop:global_Poincare_from_Ricci_diameter_and_volume}
with a uniform constant $S$.
This observation also applies to the following lemma.
\begin{lemma}[Moser iteration {\cite[Thm 9.2.7]{petersen2016riemannian}}]\label{lem:Moser_it}
  Take $\nu\in \left(1,\frac{n}{n-2}\right]$, $\lambda>0$, and let $(M,g)$ be a compact $n$-dimensional Riemannian manifold 
  such that for some $S>0$ all smooth functions on $M$ satisfy
  \begin{equation}
    \|u\|_{\overline{L}^{2\nu}(M,g)}\leq S\|du\|_{\overline{L}^{2}(M,g)}+\|u\|_{\overline{L}^{2}(M,g)}
  \end{equation}
  If $f:M\rightarrow [0,\infty )$ is continuous, smooth on $\{f>0\}$, and satisfies
  $\Delta_{g}f\geq-\lambda f$, then
  \begin{equation}
    \|f\|_{L^{\infty }}\leq\exp\left(\frac{C(S)\sqrt{\lambda\nu}}{\sqrt{\nu}-1}\right)\|f\|_{\overline{L}^{2}(M,g)}.
  \end{equation}
\end{lemma}

\section{Uniform Enlargeability}\label{sec:uniform_enlargeability}
Let $\mathcal F$ be the family of Riemannian metrics on $\mathbb T^n$ from Theorem \ref{thm:main_result-torus_stability}. In this section we will show that metrics in the family $\mathcal{F}$ are uniformly enlargeable in the following
sense. Let
\[
  \pi^k:\T^n=\R^n/(k\Z)^n\longrightarrow \T^n=\R^n/\Z^n
\]
be the canonical $k^n$-fold cover, and set $g^k=(\pi^k)^*g$. We shall show that 
there is a fixed constant $C_{\mathcal{F}}$ so that for any $g\in \mathcal{F}$ if $(\T^{n},g^{k})$ is a $k^n-$fold covering
of $(\T^{n},g)$, then there is a smooth map
\[
  \phi^{k}:(\T^{n},g^{k})\rightarrow (\Sp^{n},g_{\Sp^{n}})
\]
which has nonzero degree and satisfies
\begin{equation}
  \|d \phi^{k}\|_{L^{\infty }}\leq\frac{C_{\mathcal{F}}}{k}.
\end{equation}
Let $g_{0}$ denote the product metric on $\T^{n}$, which does have the above mentioned property.
The claim will eventually follow from showing that we can find nonzero degree maps
$(\T^{n},g)$ to $(\T^{n},g_{0})$ whose Lipschitz norms are bounded independently of
$g\in \mathcal{F}$.
We begin by showing that we can find a basis of $H^{1}(\T^{n};\Z)_{\R}$ whose
comass norms are bounded in terms of $\mathrm{stabsys}_{1}(g)$.
\begin{lemma}\label{lem:bounded_cohomology_basis}
  Let $\sigma>0$ and $g$ be a Riemannian metric on $\T^n$ such that
  \[
    \mathrm{stabsys}_1(g)\geq\sigma.
  \]
  Then there is a constant $C(n,\sigma)$ and a basis $\alpha_i$ of $H^1(\T^n;\Z)_{\R}$ with representatives 
  $a_i$ such that
    \begin{equation}
        \|a_i\|_{L^{\infty}}\leq C(n,\sigma).
    \end{equation}
\end{lemma}
\begin{proof}
    For the torus, the dual lattice of $\Lambda=H_1(\T^n;\Z)_{\R}$ with the
    mass norm $\|\cdot\|_{\mathrm{M}}$ is $\Lambda^{*}=H^1(\T^n;\Z)_{\R}$ with the comass norm
    $\|\cdot\|^{*}_{\mathrm{M}}$.
    Let $\{\lambda_{i}\}_{i=1}^{n}$ and $\{\mu_{i}\}_{i=1}^{n}$ be the successive minima of
    $H^{1}(\T^{n};\Z)_{\R}$ with respect to $\|\cdot\|^{*}_{\mathrm{M}}$ and
    $H_{1}(\T^{n};\Z)_{\R}$ with respect to $\|\cdot\|_{\mathrm{M}}$.
    By Theorem \ref{thm:bounds_dual_minima} we know that
    \begin{equation}
      \lambda_{n}\leq\frac{n!}{\mu_{1}}\leq\frac{n!}{\sigma}.
    \end{equation}
    Since by definition $\lambda_{j}\leq\lambda_{n}$ for all $j=1,\dots,n$, we have the same upper bound
    for all $\lambda_{j}$.
    Now, we may apply Corollary \ref{cor:existence_of_bounded_basis} to get the result.
\end{proof}
\begin{lemma}[Uniform Enlargeability]\label{lem:uniform_enlargeability}
  There exists a constant $C_{\mathcal{F}}(n,\sigma,D,K)=C_{\mathcal{F}}>0$ such that for all $g\in \mathcal{F}(n,\sigma,D,K)$ and all $k\in \mathbb{N}$ there 
  is a $k^{n}$-fold covering $(\mathbb{T}^{n},g^{k})$ and a map
  $\phi^{k}:(\mathbb{T}^{n},g^{k})\rightarrow (\mathbb{S}^{n},g_{\Sp^{n}})$
  such that
  \begin{equation}
    \|d\phi^{k}\|_{\infty }\leq\frac{C_{\mathcal{F}}}{k}.
  \end{equation}
\end{lemma}
\begin{proof}
  Let $(\mathrm{T}^{n},g_{0})$ be the torus with the standard product metric. 
  By Lemma \ref{lem:bounded_cohomology_basis}
  we may find a basis of integral 1-forms, say $\{a_{i}\}_{1}^{n}$, such that
  $\|a_{i}\|_{L^{\infty}(g)}\leq C\left(n,\sigma\right)$.
  Additionally, for each $i$ there is a map $u_{i}:T^{n}\rightarrow \mathbb{S}^{1}$ such that
  $du_{i}=a_{i}$. Let $\U\colon \mathbb T^n \to \mathbb T^n$ be the map $\U = (u_1, \cdots, u_n)$. 
  Since $\{du_{i}\}_{i=1}^{n}$ generates $H^{1}(T^{n};\mathbb{Z})_{\R}$, we see that
  $\mathrm{deg}(\U)=\pm1$.
  We may as well suppose that $\mathrm{deg}(\U)=1$.

  Following the proof in \cite{lawson2016spin}, but somewhat simplified here since we are only dealing with
  the torus, let $\pi^{k}:T^{n}\rightarrow T^{n}$ be the $k^{n}$-fold covering map.
  Observe that the map $\U\circ\pi^{k}:(\mathbb{T}^{n},g^{k})\rightarrow (\mathbb{T}^{n},g_{0})$
  lifts to a map $\U^{k}:(\mathbb{T}^{n},g^k)\rightarrow (\mathbb{T}^{n},g_{0}^{k})$, and that we have
  \begin{equation}
    \|d\mathbb{U}^{k}\|_{\infty }\leq C\left(n,\sigma\right).
  \end{equation}
  Letting $f:(T^{n},g_{0}^{k})\rightarrow (\mathbb{S}^{n},g_{\Sp^{n}})$ be a $\frac{1}{k}$-contracting
  map with degree 1, the result then follows by setting $\phi^{k}=f\circ\mathbb{U}^{k}$ and using the chain rule.
\end{proof}

\section{Scalar curvature bounds on harmonic 1-forms}\label{sec:scal_bounds_harmonic_1-forms}
We now analyze covering spaces to obtain geometric estimates in terms of the negative part of the scalar curvature.
Here, as above, we let
\[
  \pi^k:\T^n=\R^n/(k\Z)^n\longrightarrow \T^n=\R^n/\Z^n
\]
be the canonical $k^n$-fold cover, and set $g^k=(\pi^k)^*g$. By
Lemma \ref{lem:uniform_enlargeability} there is a constant $C_{\mathcal{F}}$ such that for
each $g\in \mathcal{F}(n,\sigma,D,K)=\mathcal{F}$ and $k\geq1$ there is a smooth map
  \[
    \varphi^{k}:(\T^n,g^{k})\longrightarrow (\Sp^n,g_{\Sp^{n}})
  \]
of nonzero degree such that
\[
  \label{eq:contracting-map}
  \mathrm{Lip}(\varphi^{k})\leq \frac{C_{\mathcal{F}}}{k}.
\]
The odd dimensional case (i.e. when $n$ is odd) of Theorem \ref{thm:main_result-torus_stability}  follows from
the even dimensional case by considering the direct product $\mathbb T^n\times \mathbb S^1$.
So, from now on we shall assume without loss of generality that $n$ is even. 

Choose an Hermitian vector bundle $E_0\to\Sp^n$ such that
\[
  \left\langle \operatorname{ch}_{n/2}(E_0),[\Sp^n]\right\rangle\neq 0.
\]
Set $E^{k}:=(\varphi^{k})^*E_0$, and let $\slashed D_{E^k}$ be the twisted Dirac operator on
$S_{\T^n}\otimes E^k$, where $S_{\T^n}$ is the spinor bundle of $(\T^n,g^k)$.
The Atiyah--Singer index theorem gives
\begin{align*}
  \operatorname{ind}(\slashed D_{E^k}^{+})
  &=\left\langle \widehat A(T\T^n)\operatorname{ch}(E^k),[\T^n]\right\rangle \\
  &=\deg(\varphi^k)
    \left\langle \operatorname{ch}_{n/2}(E_0),[\Sp^n]\right\rangle
  \neq0.
\end{align*}
In particular,
\begin{equation}
  \label{eq:nonzero-twisted-kernel}
  \ker\slashed D_{E^k}\neq\{0\}
  \qquad\text{for every }k\geq1.
\end{equation}
Furthermore, the Lichnerowicz formula states 
\[
  \slashed D_{E^k}^{\,2}
  =\nabla^*\nabla+\frac14\Scal_{g^k}+\mathfrak R^{E^k},
\]
where $\mathfrak R^{E^k}$ is the curvature endomorphism induced by $E^k$. Since curvature
is a two-form and $E^k$ is pulled back by $\varphi^k$, Equation~\eqref{eq:contracting-map}
implies the pointwise estimate
\begin{equation}
  \label{eq:curvature-term-bound}
  \|\mathfrak R^{E^k}\|
  \leq \frac{C_{\mathcal F}}{k^2}.
\end{equation}
This establishes the following result:
\begin{lemma}\label{lem:global-bochner}
For any $g\in \mathcal{F}$ the space $\ker\slashed D_{E^{k}}$ is nontrivial,
and for every element $s^{k}\in\ker\slashed D_{E^{k}}$ we have
\[
  \int_{\T^n}|\nabla s^{k}|^2dV_{g^{k}}
  \leq
  \int_{\T^n}\left( \frac14 \Scalminus_{g^{k}}+C_{\mathcal{F}}k^{-2} \right)|s_{k}|^2 dV_{ g^{k}}.
\]
\end{lemma}

We pick a normalization that compliments the averaging operation in Definition \ref{def:avg_operation} given below:
for any nonzero element $s^{k}\in \ker\slashed D_{E^{k}}$ we may rescale it, and abuse notation slightly, so that
\begin{equation}\label{eq:spinor-normalization}
  \int_{\T^n}|s^{k}|^2 dV_{g^{k}}
  =\mathrm{Vol}_{g^{k}}(\T^n)=k^{n}\mathrm{Vol}_{g}(\T^{n}).
\end{equation}
\begin{remark}
  From now on, when working with an element $s^{k}\in \ker\slashed D_{{E^{k}}}$ we will always assume that it is 
  nontrivial and normalized as above.
\end{remark}
Since $s^{k}$ is an harmonic spinor, one is tempted to apply the Poincar{\'e} inequality in
Theorem \ref{thm:Poincare_lower_bound-Ricci_lower_bound},
the Sobolev inequality which follows (Proposition \ref{prop:global_Poincare_from_Ricci_diameter_and_volume}), and Moser iteration
(Lemma \ref{lem:Moser_it})
to show that $s^{k}$ is norm bounded in terms of its $L^{2}$-norm.
The problem with doing this directly is that the constant in Theorem \ref{thm:Poincare_lower_bound-Ricci_lower_bound} grows
exponentially in terms of the diameter, and $\mathrm{diam}_{g^{k}}\left(\T^{n}\right)$ becomes arbitrarily large
as $k$ increases.

Instead, we will average over the deck transformations of $\left(\T^{n},g^{k}\right)$.
This will produce objects on $(\T^{n},g)$, which is an element of $\mathcal{F}$
and so has a known upper bound on its diameter and lower bound on its Ricci curvature.
\begin{definition}\label{def:avg_operation}
Let $\pi^{k}:\T^{n}\rightarrow \T^{n}$ be the $k^{n}$-fold covering map.
For a function $f$ on $\T^n$ define
\[
  \mathcal A_k f(x):=\frac1{k^{n}}\sum_{y\in(\pi^k)^{-1}(x)}f(y),
  \qquad x\in\T^n.
\]
where $|(\pi^k)^{-1}(x)| = k^n.$
For vector or tensor fields on the cover, we use the same notation after identifying
tangent spaces by the local isometry $d\pi^k$.
\end{definition}
The following scalar function is a crucial ingredient.
\begin{definition}\label{def:rho-definition}
  Let $g\in\mathcal{F}$ and $\pi^{k}\colon \left(\T^{n},g^{k}\right)\to\left(\T^{n},g\right)$ the covering map.
  For any nontrivial and normalized element $s^{k}\in\ker\slashed D_{E^{k}}$ define $\rho_{k}$ on
  $\T^{n}$ by
\begin{equation}\label{eq:rho-definition}
  \rho_{k}:=\mathcal A_k(|s^{k}|^2).
\end{equation}
\end{definition}
\noindent
Because of the normalization chosen for $s^{k}$, we have
\begin{equation}\label{eq:rho-mean-one}
  \frac1{\mathrm{Vol}_{g}(\T^n)}\int_{\T^n}\rho_{k} dV_{g}=1.
\end{equation}
\begin{lemma}\label{lem:nabla-rho_ptwise_estimate}
  The function $\rho_{k}$ defined above satisfies the following pointwise bound on its differential:
  \begin{equation}\label{eq:rho-gradient-pointwise}
    \lvert d\rho_{k} \rvert^{2}\leq4\rho_{k}\A\left(\lvert \nabla s^{k} \rvert^{2}\right).
  \end{equation}
\end{lemma}
\begin{proof}
  For any $X\in T_x\T^n$ (and its lift) we have
  \begin{align*}
    |d\rho_{k}(X)|
    &\leq \frac{2}{k^{n}}\sum_{i=1}^{k^{n}}
       |\nabla_{X}s^{k}|\,|s^{k}| \\
    &\leq
       2\left(\frac1{k^{n}}\sum_{i=1}^{k^{n}}|s^{k}|^2\right)^{1/2}
        \left(\frac1{k^{n}}\sum_{i=1}^{k^{n}}
        |\nabla_{X}s^{k}|^2\right)^{1/2} \\
    &=2
    \sqrt{\rho_{k}(x) \A_k(|\nabla_X s_{k}|^2)(x)},
  \end{align*}
  where $|\nabla_X s^{k}|$ denotes the covariant derivative of $s^{k}$ in the
  lifted direction. In particular, we have 
  \begin{equation}
    |d\rho_{k}|^2
    \leq 4\rho_{k}\,\A_k(|\nabla s^{k}|^2).
  \end{equation}
\end{proof}
With the above in hand, we can see that $\sqrt{\rho_{k}}$ is a subsolution.
\begin{lemma}\label{lem:rho-subsolution}
There exists a constant $C(n,\sigma,D,K)=C_{\mathcal{F}}>0$ such that for any $g\in \mathcal{F}$
the function $\rho_{k}$ defined above satisfies
\[
  \Delta_{g}\sqrt{\rho_{k}}\geq -C_{\mathcal{F}}\sqrt{\rho_{k}}.
\]
\end{lemma}

\begin{proof}
For an harmonic spinor, we have 
\[
  \frac12\Delta_{g^{k}} |s^{k}|^2
  =|\nabla s^{k}|^2+
  \left\langle\left(\frac14\Scal_{g^{k}}+\mathfrak R^{E_{k}}\right)s^{k},s^{k}\right\rangle.
\]
Using \eqref{eq:curvature-term-bound}, this gives
\[
  \frac12\Delta_{g^{k}} |s^{k}|^2
  \geq
  \lvert \nabla s^{k} \rvert^{2}
  -\left(\frac14\Scalminus_{g^{k}}+C_{\mathcal{F}}k^{-2}\right)|s_{k}|^2.
\]
Averaging gives us
\begin{equation*}
  \frac{1}{2}\Delta_{g}\rho_{k}\geq\A(\lvert \nabla s^{k} \rvert^{2})-
  \left(\frac{1}{4}\Scalminus_{g}+C_{\mathcal{F}}k^{-2}\right)\rho_{k}.
\end{equation*}
Using that $\Ric_{g}\geq -K$, we have $\Scalminus_{g}\leq nK$, and so we are led to 
\begin{equation*}
  \frac{1}{2}\Delta_{g}\rho_{k}\geq\A\left(\lvert \nabla s^{k} \rvert^{2}\right)-
  \left(\frac{nK}{4}+C_{\mathcal{F}}k^{-2}\right)\rho_{k}.
\end{equation*}
We also have that
\begin{equation}
  \Delta_{g}\sqrt{\rho_{k}}
  =
  \frac{1}{2\sqrt{\rho_{k}}}\left(\Delta_{g}\rho_{k}-\frac{\lvert d\rho^{k} \rvert^{2}}{2\rho_{k}}\right).
\end{equation}
We can now apply Lemma~\ref{lem:nabla-rho_ptwise_estimate}, the inequality above, and cancel terms to obtain
\begin{equation}
  \Delta_{g}\sqrt{\rho_{k}}\geq\frac{-1}{\sqrt{\rho_{k}}}\left(\frac{nK}{4}+C_{\mathcal{F}}k^{-2}\right)\rho_{k},
\end{equation}
which after combining terms is precisely
\begin{equation}
  \Delta_{g}\sqrt{\rho_{k}}\geq-C_{\mathcal{F}}\sqrt{\rho_{k}}.
\end{equation}
\end{proof}
This shows that $\rho_{k}$ must be uniformly bounded in $L^{\infty}$:
\begin{corollary}\label{cor:rho-linfty}
There is a constant $C_{\mathcal{F}}(n,\sigma,D,K)=C_{\mathcal{F}}>0$ such that for any $g\in \mathcal{F}$,
any $k\in \N$, and any normalized $s^{k}\in \ker\slashed D_{E^{k}}$ we have
\[
  \|\rho_{k}\|_{L^\infty(g)}\leq C_{\mathcal{F}}.
\]
\end{corollary}
\begin{proof}
  This follows directly from Lemma~\ref{lem:rho-subsolution}, Equation~\eqref{eq:rho-mean-one},
  and Moser iteration applied to $\sqrt{\rho_{k}}$. That is, since
  \begin{equation*}
    \|\sqrt{\rho_{k}}\|_{\overline{L}^{2}(g)}=
    \left(\frac{1}{\mathrm{Vol}_{g}\left(\T^{n}\right)}\int_{\T^{n}}\rho_{k}dV_{g}\right)^{\frac{1}{2}}=
    1,
  \end{equation*}
  Lemma~\ref{lem:Moser_it} now gives the desired result.
\end{proof}
We now wish to show that $d\rho_{k}$ has $L^{2}$-norm controlled by $\Scalminus_{g}$.
To do this, we will first see that the average spinor energy is bounded above by this quantity.
\begin{lemma}\label{lem:averaged-spinor-energy}
  There is a constant $C_{\mathcal{F}}(n,\sigma,D,K)=C_{\mathcal{F}}>0$ such that for all $g\in \mathcal{F}$ we have
\[
  \int_{\T^n}\A_k(|\nabla s^{k}|^2)dV_{g}
  \leq C_{\mathcal{F}}\left(\|\Scalminus_{g}\|_{L^{1}(g)}+k^{-2}\right).
\]
\end{lemma}
\begin{proof}
Divide Lemma~\ref{lem:global-bochner} by $k^{n}$ and use the averaging identities to obtain
\[
  \int_{\T^n}\A_k(|\nabla s^{k}|^2)dV_{g}
  \leq
  \int_{\T^n}\left(\frac14\Scalminus_{g}+C_{\mathcal{F}}k^{-2}\right)\rho_{k}dV_{g}.
\]
Now we may use the bound on $\|\rho_{k}\|_{L^{\infty}}$ from Corollary~\ref{cor:rho-linfty}.
\end{proof}

The above bound on the average spinor energy will control the $L^{2}$ norm of $d\rho_{k}$.
Together with the Poincar{\'e} Inequality, we will see that $\rho_{k}$ converges to $1$ in $L^{2}$-norm.
\begin{lemma}\label{lem:rho-almost-constant}
There is a constant $C_{\mathcal{F}}(n,\sigma,D,K)=C_{\mathcal{F}}>0$ such that
\[
  \int_{\T^n}|d\rho_{k}|^2 dV_{g}
  \leq C_{\mathcal{F}}\left(\|\Scalminus_g\|_{L^{1}(g)}+k^{-2}\right).
\]
and
\[
  \|\rho_{k}-1\|_{L^2(g)}
  \leq C_{\mathcal{F}}(\|\Scalminus_g\|_{L^{1}(g)}+k^{-2})^{1/2}.
\]
\end{lemma}
\begin{proof}
Integrating \eqref{eq:rho-gradient-pointwise} and using the uniform
$L^\infty$ bound from Corollary~\ref{cor:rho-linfty}, we obtain
\begin{align*}
  \int_{\T^n}|d\rho_{k}|^2 dV_{g}
  &\leq 4\|\rho_{k}\|_{L^\infty(g)}
     \int_{\T^n}\A_k(|\nabla s^{k}|^2) dV_{g} \\
  &\leq C_{\mathcal{F}}\left(\|\Scalminus_{g}\|_{L^{1}(g)}+k^{-2}\right),
\end{align*}
where the last step is Lemma~\ref{lem:averaged-spinor-energy}.

For the second part,   by
\eqref{eq:rho-mean-one} we have
\[
  \frac1{\mathrm{Vol}_{g}(\T^n)}\int_{\T^n}\rho_{k}dV_{g}=1.
\]
We apply the $L^1$ Poincar\'e Inequality (Theorem \ref{thm:Poincare_lower_bound-Ricci_lower_bound}) to $\rho_{k}-1$.
Next, using Proposition \ref{prop:global_Poincare_from_Ricci_diameter_and_volume} together with interpolation between $L^1$ and $L^2$ shows that 
\[
  \|\rho_{k}-1\|_{L^2(g)}
  \leq C_{\mathcal{F}}\|d\rho_{k}\|_{L^2(g)}.
\]
Combining this with the first part proves that
\[
  \|\rho_{k}-1\|_{L^2(g)}
  \leq C_{\mathcal{F}}(\|\Scalminus_{g}\|_{L^{1}(g)}+k^{-2})^{1/2}.
\]
\end{proof}
%
%

\subsection{Clifford multiplication by harmonic 1-forms}
In this section, we will see that harmonic spinors on the covering space, and specifically
their averaged norm $\rho_{k}$, have a close relationship with harmonic 1-forms on the base.
\begin{lemma}\label{lem:dirac-identity}
  For $g\in \mathcal{F}$, let $\omega$ be any harmonic 1-form on $(\mathbb{T}^{n},g)$.
  Set $\omega^{k}=(\pi^{k})^{*}\omega$, and for any $s^{k}\in \ker \slashed D_{E^{k}}$ define
  $\psi^{k}=\omega^{k}\cdot s^{k}$.
  Then we have
\[
  \slashed D_{E^{k}}(\psi^{k})
  =-2\nabla_{V^{k}} s^{k},
\]
where $V^{k}$ is the vector field dual to $\omega^{k}$.
In particular,
\[
  |\slashed D_{E^{k}}\psi^{k}|
  \leq 2|\omega^{k}||\nabla s^{k}|.
\]
\end{lemma}

\begin{proof}
Using that $\slashed D s^{k}=0$ and $(d+\delta)\omega^k=0$, the Clifford commutator formula gives
\begin{align*}
  \slashed D_{E^k}(\omega^k\cdot s^k)
  &=(d\omega^k+\delta\omega^k)\cdot s^k
    -\omega^k\cdot\slashed D_{E^k}s^k
    -2\nabla_{V^k}s^k \\
  &=-2\nabla_{V^k}s^k.
\end{align*}
The pointwise estimate follows from
$|\nabla_{V^k}s^k|\leq|V^k|\,|\nabla s^k|=|\omega^k|\,|\nabla s^k|$.
\end{proof}

In fact, we can bound the average $L^{2}$-norm of $\nabla\psi^{k}$ in terms of $\|\omega\|_{L^{\infty }}$
and $\|\Scalminus_{g}\|_{L^{1}}$.
\begin{lemma}\label{lem:psi-energy}
  There exists a constant $C_{\mathcal{F}}(n,\sigma,D,K)=C_{\mathcal{F}}>0$ such that for any 
  $g\in \mathcal{F}$, any harmonic 1-form $\omega$ and normalized $s^{k}$ in $\ker\slashed D_{E^{k}}$ the spinor
  $\psi^{k}=\omega^{k}\cdot s^{k}$ satisfies
\[
  \int_{\T^n}\A(|\nabla\psi_{k}|^2) dV_{g^{k}}
  \leq C_{\mathcal{F}}\|\omega\|^{2}_{L^{\infty }}(\|\Scalminus_{g}\|_{L^{1}(g)}+k^{-2}).
\]
\end{lemma}

\begin{proof}
The integrated Lichnerowicz formula applied to $\psi^{k}$ gives
\[
  \int_{\T^n}|\nabla\psi^{k}|^2dV_{g^{k}}
  \leq
  \int_{\T^n}|\slashed D\psi^{k}|^2dV_{g^{k}}
  +\int_{\T^n}\left(\frac14\Scalminus_{g^{k}}+C_{\mathcal{F}}k^{-2}\right)|\psi^{k}|^2dV_{g^{k}}.
\]
We may estimate the first term on the right by using Lemma~\ref{lem:dirac-identity}:
\[
  |\slashed D\psi^{k}|^2\leq 4\|\omega\|^{2}_{L^{\infty }(g)}|\nabla s^{k}|^2.
\]
Furthermore, we have $\lvert \psi^{k} \rvert^{2}=\lvert \omega^{k} \rvert^{2} \lvert s^{k} \rvert^{2}$.
It follows that  
\begin{equation}
  \begin{split}
    \int_{\T^n}|\nabla\psi^{k}|^2dV_{g^{k}}
    &\leq
    4\|\omega\|_{L^{\infty }}^{2}\int_{\T^{n}}\lvert \nabla s^{k} \rvert^{2}dV_{g^{k}}
    \\
    &+
    \int_{\T^n}
    \left(\frac14\Scalminus_{g^{k}}+
    C_{\mathcal{F}}k^{-2}\right)\left(\lvert \omega^{k} \rvert\lvert s^{k} \rvert\right)^2dV_{g^{k}}.
    \end{split}
\end{equation}
Since $\omega^{k}$ is the pullback of the form $\omega$ and $\Scalminus_{g^{k}}$ is the pullback of $\Scalminus_{g}$,
we have
\[
  \A_{k}\left(\left(\frac14\Scalminus_{g^{k}}+
    C_{\mathcal{F}}k^{-2}\right)\lvert \omega^{k} \rvert^{2}\lvert s^{k} \rvert^{2}\right)
  =\left(\frac{1}{4}\Scalminus_{g}+C_{\mathcal{F}}k^{-2}\right)|\omega|^{2}\A_{k}\left(\lvert s^{k} \rvert^{2}\right).
\]
So, we may average to obtain
\begin{equation}
  \begin{split}
    \int_{\T^{n}}\A_{k}\left(|\nabla \psi^{k}|^{2}\right)dV_{g}
    &\leq
    4\|\omega\|^{2}_{L^{\infty }}\int_{\T^{n}}\A_{k}(|\nabla s^{k}|^{2})dV_{g}
    \\
    &+
    \int_{\T^{n}}
    \left(\frac{1}{4}\Scalminus_{g}+C_{\mathcal{F}}k^{-2}\right)
    |\omega|^{2}\rho_{k}dV_{g}.
  \end{split}
\end{equation}
At this point we may use Lemma~\ref{lem:averaged-spinor-energy} along with
Corollary \ref{cor:rho-linfty} to obtain the result.
\end{proof}

The next estimate converts control of $\nabla\psi^k$ into a weighted estimate for $\nabla\omega$.
\begin{lemma}\label{lem:weighted-omega}
  There exists a constant $C_{\mathcal{F}}(n,\sigma,D,K)=C_{\mathcal{F}}>0$ such that for any $g\in\F$,
  any harmonic 1-form $\omega$, and any normalized $s^{k}\in \ker\slashed D_{E^{k}}$,
  we have 
  \begin{equation}
    \int_{\T^{n}}\rho_{k}|\nabla \omega|^{2}dV_{g}\leq C(n,\sigma,D,K)
    \|\omega\|^{2}_{L^{\infty }}(\|\Scalminus_g\|_{L^{1}(g)}+k^{-2}).
  \end{equation}
\end{lemma}

\begin{proof}
Recall that $\omega^k = (\pi^k)^\ast w$. Since
\[
  \nabla\psi^{k}
  = (\nabla\omega^{k})\cdot s^{k}
    +\omega^{k}\cdot\nabla s^{k},
\]
we immediately have 
\begin{equation}
  |(\nabla_{X}\omega^{k})\cdot s^{k}|^{2}\leq 2\left(|\omega^{k}|^{2}|\nabla_{X}s^{k}|^{2}+|\nabla_{X}\psi^{k}|^{2}\right).
\end{equation}
Next, recall that Clifford multiplication by unit vectors is an isometry, so that
\[
  |s^{k}|^2|\nabla_{X}\omega^{k}|^2
  \leq
  2\left(|\nabla_{X}\psi^{k}|^2
  +|\omega^{k}|^2|\nabla_{X}s^{k}|^2\right).
\]
This leaves us with
\begin{equation}
  |s^{k}|^{2}|\nabla \omega^{k}|^{2}\leq 2(|\nabla \psi^{k}|^{2}+|\omega^{k}|^{2}|\nabla s^{k}|^{2}).
\end{equation}
We now apply the averaging operation $\A$ and integrate.
Since $\omega^{k}$ is the pullback of $\omega$, the left-hand side becomes
\[
  \int_{\T^n}\rho_{k}|\nabla\omega|^2 dV_{g},
\]
while the two terms on the right are controlled by Lemma~\ref{lem:psi-energy},
the quantity $\|\omega\|_{L^{\infty }}$, and Lemma~\ref{lem:averaged-spinor-energy}.
\end{proof}

\begin{lemma}\label{lem:weighted-form-gradient}
  There exists a constant $C_{\mathcal{F}}(n,\sigma,D,K)=C_{\mathcal{F}}>0$ such that for all $g\in \F$,
  harmonic 1-forms $\omega$ on $(\T^n, g)$, $k\in\N$, and normalized $s^{k}\in \ker\slashed D_{E^{k}}$,  we have
  \begin{equation}
    \int_{\T^{n}}|\nabla(\rho_{k}\omega)|^{2}dV_{g}\leq
    C_{\mathcal{F}}\|\omega\|^{2}_{L^{\infty }}(\|\Scalminus_{g}\|_{L^{1}}+k^{-2}).
  \end{equation}
\end{lemma}

\begin{proof}
  Since
\[
  \nabla(\rho_{k}\omega)=d\rho_{k}\otimes\omega+\rho_{k}\nabla\omega,
\]
we have
\[
  |\nabla(\rho_{k}\omega)|^2
  \leq
  2|d\rho_{k}|^2|\omega|^2+2\rho_{k}^2|\nabla\omega|^2.
\]
Integrating and using $\|\omega\|_{L^{\infty }}$, Lemma~\ref{lem:rho-almost-constant}, Corollary~\ref{cor:rho-linfty},
and Lemma~\ref{lem:weighted-omega}, we get
\begin{align*}
  \int_{\T^n}|\nabla(\rho_{k}\omega)|^2 dV_{g}
  &\leq
  2\|\omega\|_{L^\infty(g)}^2
  \int_{\T^n}|d\rho_{k}|^2 dV_{g}
  +2\|\rho_{k}\|_{L^\infty(g)}
  \int_{\T^n}\rho_{k}|\nabla\omega|^2 dV_{g} \\
  &\leq C_{\mathcal{F}}\|\omega\|^{2}_{L^{\infty}(g)}(\|\Scalminus_{g}\|_{L^{1}(g)}+k^{-2}).
\end{align*}
Here we used \(\rho_{k}\geq0\), so \(\rho_k^2\leq\rho_k\|\rho_k\|_{L^\infty}\).
\end{proof}

We can now remove the weight $\rho_k$ and obtain the desired estimate for the covariant derivative
of an harmonic $1$-form.
\begin{proposition}\label{prop:unweighted-hessian}
  There exists a constant $C_{\mathcal{F}}(n,\sigma,D,K)=C_{\mathcal{F}}>0$ such that for any $g\in \F$ and any harmonic 1-form $\omega$ on $(\T^n,g)$, we have
\begin{equation}
     \label{eq:unweighted-harmonic-form-gradient} \int_{\T^n}|\nabla\omega|^2\,dV_g
  \leq C_{\mathcal F}\|\omega\|_{L^\infty}^2
  \|\Scalminus_g\|_{L^1}. 
\end{equation}
\end{proposition}
\begin{proof}
For a smooth $1$-form $\beta$ on the closed manifold $\T^n$, the integrated Bochner formula is
\begin{equation}
  \label{eq:integrated-bochner-one-form}
  \int_{\T^n}|\nabla\beta|^2\,dV_g
  =\int_{\T^n}\bigl(|d\beta|^2+|\delta\beta|^2\bigr)\,dV_g
   -\int_{\T^n}\Ric_g(\beta,\beta)\,dV_g.
\end{equation}
For each $k\geq1$ we choose a normalized $s^k\in\ker\slashed D_{E^k}$, let $\rho_k$ be
as in Definition~\ref{def:rho-definition}, and set
\[
  \alpha_k:=\rho_k\omega,
  \qquad
  \eta_k:=\alpha_k-\omega=(\rho_k-1)\omega.
\]
Since $\omega=\alpha_k-\eta_k$,
\begin{equation}
  \label{eq:omega-alpha-eta}
  \int_{\T^n}|\nabla\omega|^2\,dV_g
  \leq2\int_{\T^n}|\nabla\alpha_k|^2\,dV_g
    +2\int_{\T^n}|\nabla\eta_k|^2\,dV_g.
\end{equation}
Lemma~\ref{lem:weighted-form-gradient} controls the first term, so we need only control the second.

Because $\omega$ is harmonic, $d\eta_k=d\alpha_k$ and $\delta\eta_k=\delta\alpha_k$.
Applying Equation~\eqref{eq:integrated-bochner-one-form} to $\eta_k$, using
$|d\alpha_k|^2+|\delta\alpha_k|^2\leq C(n)|\nabla\alpha_k|^2$ and
$\Ric_g\geq-Kg$, gives
\begin{align*}
  \int_{\T^n}|\nabla\eta_k|^2\,dV_g
  &\leq C(n)\int_{\T^n}|\nabla\alpha_k|^2\,dV_g
    +K\int_{\T^n}|\eta_k|^2\,dV_g \\
  &\leq C(n)\int_{\T^n}|\nabla\alpha_k|^2\,dV_g
    +K\|\omega\|_{L^\infty}^2\|\rho_k-1\|_{L^2}^2.
\end{align*}
Lemmas~\ref{lem:weighted-form-gradient} and \ref{lem:rho-almost-constant} therefore imply that
\[
  \int_{\T^n}|\nabla\eta_k|^2\,dV_g
  \leq C_{\mathcal F}\|\omega\|_{L^\infty}^2
  \left(\|\Scalminus_g\|_{L^1}+k^{-2}\right).
\]
Substituting this and Lemma~\ref{lem:weighted-form-gradient} into
Equation~\eqref{eq:omega-alpha-eta} yields
\[
  \int_{\T^n}|\nabla\omega|^2\,dV_g
  \leq C_{\mathcal F}\|\omega\|_{L^\infty}^2
  \left(\|\Scalminus_g\|_{L^1}+k^{-2}\right)
\]
for every $k\geq1$. Letting $k\to\infty$ proves
Equation~\eqref{eq:unweighted-harmonic-form-gradient}.
\end{proof}

\section{Stability}\label{sec:stability}
Let $\mathcal F$ be the family of Riemannian metrics on $\mathbb T^n$ from Theorem \ref{thm:main_result-torus_stability}. Throughout this section, we use the uniform volume bounds
\begin{equation}\label{eq:uniform-volume-bounds}
  0<v_{\mathcal F}
  \leq \operatorname{Vol}_g(\T^n)
  \leq V_{\mathcal F}<\infty,
  \qquad g\in\mathcal F.
\end{equation}
Here the lower bound follows from the stable systolic inequality, while the upper
bound follows from the Ricci lower bound and the diameter upper bound.

\begin{lemma}
   There exists a constant
  $C_{\mathcal F}=C(n,\sigma,D,K)>0$ such that, for every
  $g\in\mathcal F$, there is a degree-one harmonic map
  \[
    \mathbb U:(\T^n,g)\longrightarrow(\T^n,g_0),
  \]
  where \(g_0\) is the standard flat metric on $\T^n$, satisfying
  \[
    \|d\mathbb U\|_{L^\infty}\leq C_{\mathcal F}.
  \]
\end{lemma}
\begin{proof}
    From Lemma \ref{lem:bounded_cohomology_basis}
    we may find closed 1-forms $a_i$ such that
    $\|a_i\|_{L^{\infty}}\leq C_{\mathcal{F}}$,
    which form a basis of $H^1(\T^n;\Z)_{\R}$.
    Let $\omega_i$ be the harmonic representative of $[a_i]$.
  Since harmonic representatives minimize the $L^2$-norm in their
  cohomology classes, \eqref{eq:uniform-volume-bounds} gives
  \[
    \|\omega_i\|_{L^2(g)}
    \leq \|a_i\|_{L^2(g)}
    \leq V_{\mathcal F}^{1/2}\|a_i\|_{L^\infty(g)}
    \leq C_{\mathcal F}.
  \]
  It follows from $\Ric_g\geq -K g$ that 
   \[
    \Delta_g|\omega_i|\geq-K|\omega_i|.
  \]
  Applying Lemma~\ref{lem:Moser_it}, together with  the lower volume bound in \eqref{eq:uniform-volume-bounds},
  yields
  \[
    \|\omega_i\|_{L^\infty(g)}
    \leq C_{\mathcal F}
       \|\omega_i\|_{\overline L^2(g)}
    \leq C_{\mathcal F}.
  \]
Because $[\omega_i]=[a_i]$ is an integral cohomology class,
  there is a  map
  $u_i:\T^n\to\R/\Z$ such that
  $du_i=\omega_i$. Each $u_i$ is harmonic. After replacing
  one basis element by its negative if necessary, we may assume that the basis
  \([\omega_1],\dots,[\omega_n]\) is positively oriented. Hence
  \[
    \mathbb U:=(u_1,\dots,u_n):\T^n\longrightarrow\R^n/\Z^n
  \]
  is harmonic and has degree one. Finally,
  \[
    |d\mathbb U|^2=\sum_{i=1}^n|\omega_i|^2,
  \]
  so the preceding \(L^\infty\)-bounds give the result.
\end{proof}
\begin{corollary}\label{cor:harmonic_small_hess}
    There exists a constant $C_{\mathcal{F}}(n,\sigma,D,K)=C_{\mathcal{F}}>0$ such that
    for any metric $g\in\F$ there is a degree 1 harmonic
    map $\U:(\T^n,g)\rightarrow(T^n,g_0)$ such that
    \begin{equation}
        \|\nabla d\U\|^{2}_{L^2}\leq
        C_{\mathcal{F}}\|\Scalminus_g\|_{L^1},
    \end{equation}
    and
    \begin{equation}
      \|d\U\|_{L^{\infty }}\leq C_{\mathcal{F}}.
    \end{equation}
\end{corollary}
\begin{proof}
    Apply the results of Section \ref{sec:scal_bounds_harmonic_1-forms}, specifically
    Proposition \ref{prop:unweighted-hessian}, to each
    of the coordinate functions of $\U$.
\end{proof}
Since we will apply one half of the classification theorem given in 
\cite{Honda-Ketterer-Christian-Mondello-Perales-Rigoni}, let us recall the required definition.
\begin{definition}[{\cite[Definition 1.1]{Honda-Ketterer-Christian-Mondello-Perales-Rigoni}}]\label{def:good_harmonics}
    A smooth map $\Phi:(M,g)\rightarrow \R^n/\Z^n$ is called $(\delta;C,\tau)$-\textit{harmonic} if the following
    conditions are satisfied.
    \begin{enumerate}[(1)]
        \item $\Phi$ is harmonic;
        \item The (averaged) energy of $\Phi$ is bounded by $C$:
              \[
              E(\Phi):=\frac{1}{2}|M|^{-1}\int_M\langle\Phi^*g_{\R^n/\Z^n},g_M\rangle dV_{g}\leq C;
              \]
        \item $\Phi$ is non-degenerate in the following sense
              \[
              D(\Phi):=\det\left(|M|^{-1}\int_M\langle d\phi_i, d\phi_j \rangle dV_{g}\right)_{ij}\geq\tau>0;
              \]
        \item For $i=1,2,\dots,n$, we have
              \[
              |M|^{-1}\int_{M}|\nabla d\phi_i|^2 dV_{g}\leq \delta.
              \]
    \end{enumerate}
\end{definition}
We now state and prove the stability result.
\begin{theorem}
    For $n\in\mathbb{N}$, $\sigma>0$, $D<\infty$ and $K<\infty$,
  let $\mathcal F = \mathcal{F}(n,\sigma, D,K)$ denote the family of Riemannian metrics on $\mathbb{T}^{n}$ such that
  \begin{enumerate}[(a)]
    \item $\mathrm{stabsys}_{1}(g)\geq\sigma$;
    \item $\mathrm{diam}(g)\leq D$;
    \item $\mathrm{Ric}_{g}\geq-Kg$.
  \end{enumerate}
  Then, for any sequence $g_{i}\in \mathcal{F}$ such that $\|R^{-}_{g_{i}}\|_{L^{2}}\rightarrow0$, there is a subsequence
  $g_{i_{j}}$ and a flat metric $g_{\infty}$ such that
  \begin{equation}
    \lim_{j\rightarrow \infty }d_{\mathrm{mGH}}\left(
      \left(\mathbb{T}^{n},d_{g_{i_{j}}},\mathrm{vol}_{g_{i_{j}}}\right),
      \left(\mathbb{T}^{n},d_{g_{\infty }},\mathrm{vol}_{g_{\infty }}\right)
    \right)=0.
  \end{equation}
\end{theorem}
\begin{proof}
  Write
  \[
    \varepsilon_i:=\|\Scalminus_{g_i}\|_{L^1(g_i)}.
  \]
  The uniform upper volume bound in \eqref{eq:uniform-volume-bounds} and
  H\"older's inequality give
  \[
    \varepsilon_i
    \leq V_{\mathcal F}^{1/2}
       \|\Scalminus_{g_i}\|_{L^2(g_i)}
    \longrightarrow0.
  \]

  Recall that $g_{0}$ denotes the standard flat metric on $\T^n$.
  By Corollary~\ref{cor:harmonic_small_hess}, for every $i$ there is a
  degree-one harmonic map
  \[
    \mathbb U_i=(u_i^1,\dots,u_i^n):
    (\T^n,g_i)\longrightarrow(\T^n,g_0)
  \]
  such that
  \begin{equation}\label{eq:Ui-uniform-bounds}
    \|d\mathbb U_i\|_{L^\infty(g_i)}\leq C_{\mathcal F}
    \textup{ and }
    \|\nabla d\mathbb U_i\|_{L^2(g_i)}^2
    \leq C_{\mathcal F}\varepsilon_i.
  \end{equation}
  In particular, the lower volume bound in \eqref{eq:uniform-volume-bounds}
  implies that
  \[
    E(\mathbb U_i)
    =\frac12 \mathrm{Vol}_{g_i}^{-1} \int_{\T^n}|d\mathbb U_i|^2\,dV_{g_i}
    \leq\frac{(C_{\mathcal F})^2}{2 v_{\mathcal F}}
  \]
  and, for each $\alpha=1,\dots,n$,
  \begin{equation}\label{eq:Ui-averaged-hessian}
    \mathrm{Vol}_{g_i}^{-1} \int_{\T^n}|\nabla du_i^\alpha|^2\,dV_{g_i}
    \leq \frac{C_{\mathcal F}}{v_{\mathcal F}}\varepsilon_i
    =:\delta_i,
  \end{equation}
  where $\delta_i\to0$.

  It remains to establish a uniform lower bound for
  $D(\mathbb U_i)$. Set
  \[
    A_i:=\left(
     \mathrm{Vol}_{g_i}^{-1} \int_{\T^n}
      \langle du_i^\alpha,du_i^\beta\rangle\,dV_{g_i}
    \right)_{\alpha,\beta=1}^n.
  \]
  Fix $c\in\R^n$ with $|c|=1$, choose $Q\in\mathrm{SO}(n)$ whose
  first row is $c$, and define
  \[
    \theta_i^\alpha:=\sum_{\beta=1}^nQ_{\alpha\beta}\,du_i^\beta,
    \qquad \alpha=1,\dots,n.
  \]
  Since $\deg(\mathbb U_i)=1$, the standard flat torus has volume one,
  and $\det Q=1$, we have 
  \[
    1
    =\left|\int_{\T^n}
      \theta_i^1\wedge\cdots\wedge\theta_i^n\right|.
  \]
  Also, $\|\theta_i^\alpha\|\leq\|d\mathbb U_i\|_{L^\infty}\leq C_{\mathcal F}$. Hence
  \begin{align*}
    1
    &\leq
      \int_{\T^n}\|\theta_i^1\|\cdots\|\theta_i^n\|\,dV_{g_i} \\
    &\leq
      (C_{\mathcal F})^{n-1}\operatorname{Vol}_{g_i}(\T^n)^{1/2}
      \|\theta_i^1\|_{L^2(g_i)}.
  \end{align*}
  Consequently,
  \[
    c^TA_ic
    =\mathrm{Vol}_{g_i}^{-1}\int_{\T^n}|\theta_i^1|^2\,dV_{g_i}
    \geq
    \frac{1}{(C_{\mathcal F})^{2(n-1)}\operatorname{Vol}_{g_i}(\T^n)^2}
    \geq
    \frac{1}{(C_{\mathcal F})^{2(n-1)}V_{\mathcal F}^2}
    =:\lambda_0>0.
  \]
  Since this holds for every unit vector $c$, we have
  $A_i\geq\lambda_0 I$, and therefore
  \begin{equation}\label{eq:Ui-nondegenerate}
    D(\mathbb U_i)=\det A_i\geq\lambda_0^n=:\tau>0.
  \end{equation}

  Equations \eqref{eq:Ui-averaged-hessian} and
  \eqref{eq:Ui-nondegenerate}, together with the energy bound above, show
  that $\mathbb U_i$ is
  $(\delta_i;\frac{(C_{\mathcal F})^2}{2 v_{\mathcal F}},\tau)$-harmonic in the sense of
  Definition~\ref{def:good_harmonics}. Since $\delta_i\to0$,
  \cite[Theorem~1.4(1)]{Honda-Ketterer-Christian-Mondello-Perales-Rigoni}
  provides flat $n$-tori $(T_i,h_i')$ and numbers $\Psi_i\to0$ such
  that
  \[
    d_{\mathrm{GH}}\bigl((\T^n,g_i),(T_i,h_i')\bigr)\leq\Psi_i.
  \]
  For all sufficiently large $i$, the degree-one conclusion in
  \cite[Theorem~1.4(1c)]{Honda-Ketterer-Christian-Mondello-Perales-Rigoni}
  also gives an affine harmonic diffeomorphism
  \[
    H_i:(\T^n,g_0)\longrightarrow(T_i,h_i')
  \]
  whose bi-Lipschitz constant is bounded uniformly in $i$, and for which
  $H_i\circ\mathbb U_i$ is a $\Psi_i$-Gromov--Hausdorff approximation.

  Pull back $h_i'$ by $H_i$ and write
  $h_i:=H_i^*h_i'$. Because $H_i$ is affine, $h_i$ is a
  translation-invariant flat metric on the fixed torus $\T^n$. The uniform
  bi-Lipschitz bound implies
  \[
    C^{-1}g_0\leq h_i\leq Cg_0
  \]
  for a constant independent of $i$. Thus, after passing to a subsequence,
  the corresponding positive-definite coefficient matrices converge to a
  positive-definite matrix, and hence
  \[
    h_i\longrightarrow g_\infty
  \]
  smoothly for some flat metric $g_\infty$ on $\T^n$. It follows that
  \[
    d_{\mathrm{GH}}\bigl((\T^n,g_i),(\T^n,g_\infty)\bigr)
    \longrightarrow0.
  \]

  Finally, the lower volume bound in \eqref{eq:uniform-volume-bounds},
  together with the Ricci lower and diameter upper bounds, makes the sequence
  noncollapsed. The volume convergence theorem for noncollapsed sequences
  with a uniform Ricci lower bound therefore upgrades the preceding
  Gromov--Hausdorff convergence to measured Gromov--Hausdorff convergence
  with the Riemannian volume measures. This proves the theorem.
\end{proof}

\section{Stability of the Positive Mass Theorem for Riemannian spin manifolds with a lower Ricci curvature bound} \label{sec:pmt-stability}
A sticking point for establishing the stability of the Positive Mass Theorem in 
higher dimensions is the lack of a formula resembling that of \cite{Bray-Kazaras-Khuri-Stern}.
This is resolved in \cite{PMT-stab_Kaehler} for K{\"a}hler manifolds.
Instead of working with harmonic functions and Hessians, we will use the mass formula
for spin Riemannian manifolds to work with vector fields and their covariant derivatives.

Let us first briefly recall the relevant definitions, with modifications to fit the $n$-dimensional case.
\begin{definition}[{cf. \cite[Definition 1.1]{Kazaras:2021rfv}}]\label{def:b-tau-asymflat}
A smooth connected Riemannian $n$-manifold $(M,g)$ is
\emph{asymptotically flat with finitely many ends} if there is a compact set
$\mathcal K\subset M$, an integer $N\ge0$, and connected components
\[
 M\setminus\mathcal K
 =\bigsqcup_{\lambda=0}^{N}\mathcal E_\lambda
\]
with diffeomorphisms
\[
 \Phi_\lambda:\mathcal E_\lambda\longrightarrow
 \R^n\setminus\overline B^{\mathbb E}_1(0).
\]
For each $\lambda$,  there are constants $b_\lambda>0$ and
$\tau_\lambda>(n-2)/2$ such that
\begin{equation}
\sum_{|\beta|\le2}
 \bigl|\partial^\beta(g^{(\lambda)}_{ij}-\delta_{ij})\bigr|
 \le b_\lambda |x|^{-\tau_\lambda-|\beta|}
\end{equation}
for the metric
$g^{(\lambda)}:=\Phi_{\lambda*}g$, where $|x|$ denotes the Euclidean norm of $x$, $\beta$ is any multi-index, and $|\beta|$ is its size. In this case, the end $\mathcal E_\lambda$ is said to be $(b_\lambda, \tau_\lambda)$-asymptotically flat.  Furthermore, we require the scalar curvature $R_g$ to be $L^1$-integrable over $M$. The map $\Phi$ is part of the data. 
\end{definition}
We also recall the useful weighted $L^{p}$ spaces.
\begin{definition}[{see \cite[Definition A.20]{lee2021geometric}}]
	Let $(M,g)$ be a complete asymptotically flat manifold, and let $r$ be a smooth positive functions such that
	in each asymptotically flat coordinate chart we have $r=r(x)=|x|$.
	Given any $p\geq1$ and $s\in \R$ we define the weighted Lebesgue space $L^{p}_{s}(M)$ to be the space of
	all functions $u\in L^{p}_{loc}$ with finite weighted norm
	\begin{equation}
		\|u\|_{L^{p}_{s}}=\left(\int_{M}r^{-sp-n}|u|^{p}dV_{g}\right)^{\frac{1}{p}}.
	\end{equation}
	Similarly, we define the weighted Sobolev space $W^{k,p}_{s}(M)$ to be the space of all functions
	$u\in W^{k,p}_{loc}(M)$ with finite norm
	\begin{equation}
		\|u\|_{W^{k,p}_{s}(M)}=\sum_{i=0}^{k}\|\nabla^{i}u\|_{L^{p}_{s-i}(M)}.
	\end{equation}
	In fact, this can be extended to any smooth vector bundle $E\rightarrow M$ with a metric structure.
	In particular, we can define $W^{1,p}_{s}(S(M))$ for the spin bundle in the case that $M$ is a
	spin manifold.
\end{definition}

\begin{theorem}\label{thm:pmt-stability}
	Fix an integer $n\ge3$, constants $b>0$, $\tau>(n-2)/2$ and $K\ge0$,
	and a point
	\[
	\mathbf p\in\R^n\setminus\overline B^{\mathbb E}_1(0).
	\]
	Let $(M_i,g_i)$ be a sequence of $n$ dimensional, connected, complete, asymptotically flat
  spin manifolds without boundary, each having finitely many ends
  $\mathcal E_{i,0},\ldots,\mathcal E_{i,N_i}$.  Suppose the distinguished end
  $\mathcal E_{i,0}$ has a chart $\Phi_i:=\Phi_{i,0}$ which is
  $(b,\tau)$-controlled.
  Assume
	\begin{equation}\label{eq:curvature-assumptions}
		\Scal_{g_i}\ge0,
		\qquad
		\Ric_{g_i}\ge-K,
	\end{equation}
	and set $p_i:=\Phi_i^{-1}(\mathbf p)$.  If the ADM mass of the distinguished end $\mathcal E_{i, 0}$  converges to zero,
	then 	$(M_i,d_{g_i},p_i)$ converges to the Euclidean space $(\R^n,d_{\mathbb E},0)$ in the pointed Gromov-Hausdorff sense.
\end{theorem}

To prove the theorem, we combine Witten’s spinorial approach to the positive mass theorem with techniques from the theory of RCD spaces. To streamline the argument, we divide the proof into a sequence of lemmas and propositions. Some of these are fairly standard adaptations of techniques from the theory of RCD spaces to the present setting. For the reader’s convenience, we provide brief justifications, while referring to the cited sources for full details.

In the following, unless otherwise specified, for a connected, complete, spin $n$-dimensional manifold $(M, g)$  without boundary, with the
finitely many asymptotically flat ends
$\mathcal E_0,\ldots,\mathcal E_N$ as in \cref{def:b-tau-asymflat}, we denote by $D$ its associated Dirac operator and $ m_{\mathrm{ADM}}(\mathcal E_\lambda, g)$ the ADM mass of the end $\mathcal E_\lambda$. We first construct $n$ almost parallel and almost linearly independent vector fields on $(M, g)$ in \cref{lem:spinor_properties,prop:almost-parallel-vector-fields,lem:Gram-control}. 

 \begin{lemma}\label{lem:spinor_properties}
 	Let $(M,g)$ be connected, complete, spin, and without boundary, with the
finitely many asymptotically flat ends
$\mathcal E_0,\ldots,\mathcal E_N$ from \cref{def:b-tau-asymflat}.  Let $\sigma_\lambda$ be a
constant spinor at infinity of the end $\mathcal E_\lambda$.  If
$\Scal_g\ge0$, then there is a unique harmonic spinor $\psi$ (that is, $D\psi = 0$) satisfying
\begin{enumerate}[(a)]
 		\item $\psi = \sigma_\lambda$ at infinity of each end $\mathcal E_\lambda$;
        \item $\|\psi\|_{L^{\infty }}\leq \max_{0\le\lambda\le N}|\sigma_\lambda|$;
 		\item and
        \[ \int_{M}|\nabla \psi|^{2}+\frac{1}{4}\Scal|\psi|^{2}=c_n\sum_{\lambda=0}^{N}
 m_{\mathrm{ADM}}(\mathcal E_\lambda,g)|\sigma_\lambda|^2.
 \]
 	\end{enumerate}
    where $c_n>0$ only depends on the dimension $n$. 
 \end{lemma}
 \begin{proof}
 	This result is essentially \cite[Theorem 4.1 and Theorem 4.2]{Parker-Taubes_Wittens_proof},
 	however the authors only state the result for dimension 3. Furthermore, the lemma above
  is a special case where, following their notation,
  we are assuming that the tensor $h$ is identically zero.
  They show that a unique harmonic tensor $\psi$ exists and has the property that for any
  $p>n$ we have $\psi-\sigma_{\lambda}\in W^{1,p}_{\tau}$ for each end $\mathcal{E}_{\lambda}$.
  With this in hand, one can use Kato's inequality together with a weighted Sobolev embedding
  theorem, see \cite[Theorem A.25]{lee2021geometric}, to conclude that
  for each end we have $\lim_{|x|\to\infty}|\psi|=|\sigma_{\lambda}|$.
  Finally, it follows from the Bochner formula and the scalar curvature lower bound that
  $|\psi|^{2}$ is a subsolution, and so the maximum principle applies.

  This argument extends to higher dimensions after making the following observations.
  First, the estimates on $\psi$ given in \cite{Parker-Taubes_Wittens_proof} are split 
  into two pieces: an interior and an asymptotic piece. Second,
  the estimates for the interior piece remain unchanged, while the correct estimates for the
  exterior piece follow from the results in \cite{Choquet-Bruhat-Christodoulou-81}.
  Crucially the argument that $D$ is injective in the appropriate weighted space
  also extends to higher dimensions.
  Together, this gives the required estimates.
  Finally, for the existence of the harmonic spinor and the validity of the mass formula,
  one can look to \cite[Chapter 5]{lee2021geometric}.
 \end{proof}
\begin{remark}
    A similar result is established in \cite[Section 3]{ShiTam} in detail
    for lower regularity metrics and with the falloff $\tau=n-2$.
    In this case, $\psi-\sigma_{\lambda}$ has a slightly different
    decay rate, but it is still more than sufficient to conclude that
    $\lim_{|x|\to\infty}|\psi|=|\sigma_{\lambda}|$.
\end{remark} 

\begin{proposition}\label{prop:almost-parallel-vector-fields}
	Under the same hypotheses as \cref{lem:spinor_properties},
    let $e_1, \ldots, e_n$ be coordinate vectors at infinity of the distinguished end $\mathcal E_0$.
  Then there are smooth vector fields $W^1,\ldots,W^n$ on $M$ such that
	\begin{enumerate}[$(1)$]
		\item 	$W^a  = e_a$ at infinity of $\mathcal E_0$,
        \item $W^a  = 0$ at infinity of $\mathcal E_\lambda$ for each $\lambda >0$,
		\item $|W^a| \le1$,
		\item and 
		\[ \int_M|\nabla W^a|^2dV_g \le 4c_n \cdot m_{\mathrm{ADM}}(\mathcal E_0, g). \]
	\end{enumerate}
\end{proposition}
\begin{proof}
	For a spinor $\psi$, define a corresponding vector field  $W_\psi$
	by
	\begin{equation}\label{eq:Dirac-current}
		g(W_\psi,X)=\operatorname{Im}\langle\psi,c(X)\psi\rangle.
	\end{equation}
	where $c(X)$ is the Clifford multiplication by $X$. 
	Clearly, we have 
	\[  |W_\psi|\le |\psi|^2,
	\qquad
	|\nabla W_\psi|\le2|\psi|\,|\nabla\psi|. \]
    Observe that, for a unit vector $v\in \mathbb R^n$,  $i c(v)$ is a
	self-adjoint involution.  Let $\sigma$ be  a unit eigen-spinor of  $ic(v)$ with eigenvalue $1$. Note that, if $w\perp v$, then
	$ic(v)$ anti-commutes with $c(w)$, so
	$c(w)\sigma$ is
	orthogonal to $\sigma$. If we define  $W_{\sigma}$ according to \cref{eq:Dirac-current}, then  $W_{\sigma} = v$. It follows from this discussion that, for each  coordinate vector $e_a$ at infinity of the distinguished end $\mathcal E_0$, there is a  spinor $\sigma_a$ at infinity  such that  $W_{\sigma_a}=e_a$.  Let $\psi_a$ be the corresponding spinor such that $\psi_a = \sigma_a$ at infinity of $\mathcal E_0$ and $0$ at infinith of other ends $\mathcal E_\lambda$ (see \cref{lem:spinor_properties}) and 
  define  $W^a\coloneqq W_{\psi_a}$.
	
Note that 
	\[
	\int_M|\nabla W^a|^2dV_g
	\le4\int_M|\nabla\psi_a|^2dV_g.
	\]
The proposition now follows from \cref{lem:spinor_properties}.
\end{proof}

\begin{lemma}\label{lem:Gram-control}
  Fix $n,b,\tau,K$ and $\mathbf p$ as in \cref{thm:pmt-stability}.
  Let $(M,g)$ be connected, complete,
  spin, and without boundary, with finitely many asymptotically flat ends, and
  suppose the distinguished end $\mathcal E_0$ with chart $\Phi_0$ is
  $(b,\tau)$-asymptotically flat.  Assume
  $\Scal_g\ge0$ and $\Ric_g\ge-K$, set
  $p=\Phi_0^{-1}(\mathbf p)$, and let $W^1,\ldots,W^n$ be the vector fields  from
	\cref{prop:almost-parallel-vector-fields}.
	Then for every $R>0$, we have 
	\begin{equation}\label{eq:Gram-control}
		\int_{B_R(p)}|g(W^a,W^b)-\delta_{ab}|^2dV_g
		\le C_R\cdot m_{\mathrm{ADM}}(M,g)
	\end{equation}
  for some constant $C_R=C(R,n,b,\tau,K,\mathbf p)<\infty$.
\end{lemma}

\begin{proof}
	By $(1)$ and $(2)$ of \cref{prop:almost-parallel-vector-fields}, we have $h^{ab}\to0$ at infinity and 
	\[
	|\nabla h^{ab}|
	\le |\nabla W^a|+|\nabla W^b|.
	\]
Now a standard application of Poincar\'e inequality together with $(3)$ of \cref{prop:almost-parallel-vector-fields} implies  \eqref{eq:Gram-control}.
\end{proof}

\subsection{Limits of almost-parallel vector fields}\label{sec:frame-limit}

Applying the preceding construction to each
$(M_i,g_i,\Phi_i)$ in \cref{thm:pmt-stability}, we obtain vector
fields $
W_i^1,\ldots,W_i^n$
that are uniformly bounded and, on every fixed ball,
almost parallel and almost orthonormal; see
\cref{prop:almost-parallel-vector-fields,lem:Gram-control}.  We now use the theory
of noncollapsed Ricci-limit spaces to analyze the behavior of these vector
fields as $i\to\infty$.

 In the following, for a metric space $(X, d)$, let $\cH^n$ denote the $n$-dimensional Hausdorff measure associated to the metric $d$ on $X$. We assume familiarity with $L^2$-tangent module $L^2(TX)$ associated to an $\RCD$ space $(X,d,\mathfrak m)$. We refer the reader to \cite{Gigli_2018} and \cite{gigli2020lectures} for more details. The space $L^2_{\mathrm{loc}}(TX)$ consists of measurable tangent fields
whose squared norm is integrable on every bounded ball. For a vector field
$V\in L^2_{\mathrm{loc}}(TX)$,  we write
$V\in D_{\mathrm{loc}}(\diver)$ if there is a function
$\diver V\in L^2_{\mathrm{loc}}(X)$ such that
\[
\int_X\langle V,\nabla\phi\rangle\,d\mathfrak m
=-\int_X\phi\,\diver V\,d\mathfrak m
\qquad\text{for every }\phi\in\Lip_c(X).
\] 
where $\Lip_c(X)$ denotes the compactly supported
Lipschitz functions. 
If $X$ is an $\RCD(k,n)$ space, then we  denote by $W^{1,2}_{\mathrm{loc}}(TX)$ the space of vector fields
$V\in L^2_{\mathrm{loc}}(TX)$ admitting a weak covariant derivative $
\nabla V\in L^2_{\mathrm{loc}}(T^{\otimes 2}X)$
in the sense of \cite[Definition 3.4.1]{Gigli_2018}.  
\begin{proposition}
	\label{prop:vector-Rellich}
	Let $(N_i,h_i,q_i)$ be complete pointed $n$-dimensional Riemannian
	manifolds satisfying $
	\Ric_{h_i}\ge -K$ and $
	\vol_{h_i}(B_1(q_i))\ge v>0$. Suppose that 	$(N_i,d_{h_i},q_i,dV_{h_i})$ converges to $(X,d,q,\cH^n)$ in the pointed measured Gromov-Hausdorff sense.
	Let $W_i$ be a smooth vector field on $N_i$ such that, for every $R>0$,
	there are constants $L_R,E_R<\infty$, independent of $i$, for which
	\begin{equation}\label{eq:Rellich-bounds}
		\|W_i\|_{L^\infty(B_R(q_i))}\le L_R
		\textup{ and }
		\int_{B_R(q_i)}|\nabla W_i|^2\,dV_{h_i}\le E_R.
	\end{equation}
	Then, after passing to a subsequence, there exists  $
	W\in L^2_{\loc}(TX) $
	such that the following are satisfied. 
	\begin{enumerate}[$(1)$]
		\item $W_i$ locally $L^2$-converges strongly to $W$ in the sense of Honda \cite{Honda_2015}. More precisely, on every ball $B_R(q)\subset X$, we have $W_i|_{B_R(q_i)}$   $L^2$-converges strongly to $W|_{B_R(q)}$ in the sense of Honda \cite{Honda_2015}.
		\item If $0<R<R'$ and $ 
		\sup_i\|W_i\|_{L^\infty(B_{R'}(q_i))}\le L,$
		then $\|W\|_{L^\infty(B_{R}(q))}\le L$. 
		\item $W\in W^{1,2}_{\mathrm{loc}}(TX)$ and 
		$\nabla W_i$ locally $L^2$-converges weakly to $\nabla W$.  Moreover, for every $R>0$
		one has
		\begin{equation}\label{eq:covariant-lsc}
			\int_{B_R(q)}|\nabla W|^2\,d\cH^n
			\le
			\liminf_{i\to\infty}
			\int_{B_R(q_i)}|\nabla W_i|^2\,dV_{h_i}.
		\end{equation}
	\end{enumerate}

\end{proposition}
\begin{proof}
	This is essentially a direct consequence of \cite[Theorem~6.9]{Honda_2018}.
  For the reader's convenience, we sketch a proof. 
	
	 Choose radii $ R_j = 2j$.
	For a given $j$,  Bishop--Gromov comparison theorem  and the first inequality in
	\eqref{eq:Rellich-bounds} imply 
	\[
	\sup_i
	\|W_i\|_{L^4(B_{R_{j+1}}(q_i))}
	<\infty.
	\]
	Together with the second inequality in \eqref{eq:Rellich-bounds}, this
	yields
	\begin{equation}\label{eq:local-W4-bound}
		\sup_i\left(
		\|W_i\|_{L^4(B_{R_{j+1}}(q_i))}
		+
		\|\nabla W_i\|_{L^2(B_{R_{j+1}}(q_i))}
		\right)<\infty.
	\end{equation}
	
	It follows from 
	\cite[Proposition~3.50]{Honda_2015}, applied with exponent $4$ and
	tensor type $(1,0)$, that  there is a subsequence and a vector field
	\[
	W^{(j)}\in L^4(TB_{R_j}(q))
	\]
	such that
	\[
	W_i\to  W^{(j)}
	\qquad\text{weakly in }L^4
	\text{ on }B_{R_j}(q).
	\]
	Choosing the subsequences successively and taking a diagonal
	subsequence, we may assume that this convergence holds for every
	$j$.  Uniqueness of weak limits on overlapping balls gives
	\[
	W^{(j+1)}|_{B_{R_j}(q)}=W^{(j)}
	\textup{ almost everywhere.}
	\]
	Hence these vector fields define a single vector field
	\[
	W\in L^4_{\loc}(TX)\subset L^2_{\loc}(TX)
	\]
	such that
	\[
	W_i\to  W
	\qquad\text{weakly in }L^4
	\text{ on every }B_{R_j}(q).
	\]
	
	We now improve this convergence to strong $L^2$ convergence.  Note that the volume noncollapsing condition $\vol_{h_i}(B_1(q_i))\ge v>0$ implies that the convergence $(N_i,d_{h_i},q_i,dV_{h_i}) \to (X,d,q,\cH^n)$ is noncollapsed.
	The estimate \eqref{eq:local-W4-bound} verifies the hypotheses of
	\cite[Theorem~6.9]{Honda_2018} with $p=2$ and tensor type $(1,0)$.
	
	Although \cite[Theorem~6.9]{Honda_2018} is stated for compact
	ambient Ricci-limit spaces, its proof is local: it uses harmonic
	rectifiable charts and estimates on balls compactly contained in the
	domain. See also \cite[Remark 3.76]{Honda_2015}.   Applying that argument on $B_{R_{j+1}}(q_i)$ and restricting
	to $B_{R_j}(q_i)$ therefore gives
	\[
	W_i\to  W
	\qquad\text{strongly in }L^2
	\text{ on }B_{R_j}(q).
	\]
	In particular, 
	\[
	\lim_{i\to\infty}
	\int_{B_{R_j}(q_i)}|W_i|^2\,dV_{h_i}
	=
	\int_{B_{R_j}(q)}|W|^2\,d\cH^n.
	\]
	
	For $R>0$, 
	choose $j$ with $R<R_j$.  Restricting the preceding convergence to
	$B_R(q)$ gives
	\[
	W_i\longrightarrow W
	\qquad\text{strongly in }L^2
	\text{ on }B_R(q).
	\]
	This proves item $(1)$.
	
	Let us now prove item $(2)$.
	Suppose that $0<R<R'$ and
	\[
	|W_i|\le L
	\qquad\text{a.e. on }B_{R'}(q_i).
	\]
	We write
	\[
	\Lip_c(U)
	:=
	\{\phi\in\Lip(X):\supp\phi\subset U\}.
	\] Let
	\[
	0\le\phi\in\Lip_c(B_R(q)),
	\]
	and choose nonnegative Lipschitz approximations $\phi_i$ such that
	\[
	\phi_i\longrightarrow\phi
	\quad\text{strongly},
	\qquad
	\supp\phi_i\subset B_{R'}(q_i)
	\]
	for all sufficiently large $i$.  The strong $L^2$ convergence of
	$W_i$ gives
	\[
	\int \phi|W|^2\,d\cH^n
	=
	\lim_{i\to\infty}
	\int \phi_i|W_i|^2\,dV_{h_i}.
	\]
	Since $\phi_i\ge0$ and $|W_i|\le L$ on $\supp\phi_i$,
	\[
	\int \phi_i|W_i|^2\,dV_{h_i}
	\le
	L^2\int\phi_i\,dV_{h_i}.
	\]
	Passing to the limit yields
	\[
	\int\phi\bigl(|W|^2-L^2\bigr)\,d\cH^n\le0.
	\]
	Since this holds for every nonnegative
	$\phi\in\Lip_c(B_R(q))$, we conclude that
	\[
	|W|\le L
\textup{ almost everywhere  on }B_R(q).
	\]
	This prove item $(2)$. 
	
	Now item $(3)$ is the consequence of a local version of
	\cite[Theorem~2.16]{Honda-Ketterer-Christian-Mondello-Perales-Rigoni}.  Compare also \cite[Remark~2.3]{Honda-Ketterer-Christian-Mondello-Perales-Rigoni}.   Applying \cite[Theorem~2.16]{Honda-Ketterer-Christian-Mondello-Perales-Rigoni}
	on a slightly larger ball gives local weak convergence of the
	covariant derivatives and \eqref{eq:covariant-lsc}.  An exhaustion by
	 balls proves item $(3)$.
\end{proof} 

By combining  \cref{prop:vector-Rellich} with \cref{prop:almost-parallel-vector-fields} and \cref{lem:Gram-control}, we have the following. 

\begin{proposition}[Limit of an almost-parallel orthonormal frame]\label{prop:limit-frame}
	Let $(N_i,h_i,q_i)$ be complete pointed $n$ dimensional Riemannian
	manifolds satisfying $\Ric_{h_i}\ge -K$ and $\vol_{h_i}(B_1(q_i))\ge v>0$.
  Suppose that 	$(N_i,d_{h_i},q_i,dV_{h_i})$ converges to $(X,d,q,\cH^n)$ in the pointed measured Gromov-Hausdorff sense.
  Furthermore, suppose there are smooth vector fields $W_i^1,\ldots,W_i^n$ with $\|W_i^a\|_{L^\infty}\le L$ such that for every $R>0$
  we have
	\begin{align}
		\int_{B_R(q_i)}|\nabla W_i^a|^2dV_{h_i}&\longrightarrow0,
		\label{eq:frame-energy-assumption}\\
		\int_{B_R(q_i)}
		\bigl|h_i(W_i^a,W_i^b)-\delta_{ab}\bigr|^2dV_{h_i}
		&\longrightarrow0
		\label{eq:frame-Gram-assumption}
	\end{align}
	for all $1\leq a, b\leq n$. 
	Then, after passage to a subsequence, $W_i^a$  converges strongly in
	$L^2_{\mathrm{loc}}$ to $W^a\in L^2_{\mathrm{loc}}(TX)$, where  $ W^a\in  W^{1,2}_{\mathrm{loc}}(TX)
	\cap D_{\mathrm{loc}}(\diver)$
	satisfies 
	\begin{enumerate}[$(1)$]
		\item $\nabla W^a=0$,
		\item $\diver W^a =0$,
		\item $\langle W^a,W^b\rangle=\delta_{ab}$
		 almost everywhere,
		 \item and $\|W^a\|_{L^\infty}\le L$.
	\end{enumerate}
	Moreover, the local dimension of $X$ equals $n$ almost everywhere, and the vectors $W^1(x),\ldots,W^n(x)$ form an
	orthonormal basis of the tangent fiber $T_xX$ for
	almost every $x$.  
\end{proposition}

\begin{proof}
	It follows from \cref{prop:vector-Rellich} that, after passing to a
	subsequence,
	\[
	W_i^a\longrightarrow W^a
	\qquad\text{strongly in }L^2_{\mathrm{loc}},
	\qquad a=1,\ldots,n,
	\] 
	for $W^a\in L^2_{\mathrm{loc}}(TX)$ with $\|W^a\|_{L^\infty}\le L$. Moreover,  for every  
	$R>0$, 
	one has
	\[
	\int_{B_R(q)}|\nabla W^a|^2\,d\cH^n
	\le
	\liminf_{i\to\infty}
	\int_{B_R(q_i)}|\nabla W_i^a|^2\,dV_{h_i}
	=0.
	\]
	Thus $	W^a\in W^{1,2}_{\mathrm{loc}}(TX)$ and $	\nabla W^a=0$.

	Set
	\[
	G_i^{ab}:=h_i(W_i^a,W_i^b),
	\qquad
	G^{ab}:=\langle W^a,W^b\rangle.
	\]
	By \cite[Corollary~3.59 and Proposition~3.62]{Honda_2015}, the strong local $L^2$ convergence of $W_i^a$ implies
	\[
	G_i^{ab}\longrightarrow G^{ab}
	\qquad\text{strongly in }L^1_{\loc}.
	\]
	On the other hand, \eqref{eq:frame-Gram-assumption} and the uniform
	upper volume bounds on fixed balls give
	\[
	G_i^{ab}\longrightarrow\delta_{ab}
	\qquad\text{strongly in }L^1_{\loc}.
	\]
	Consequently,
	\[
	\langle W^a,W^b\rangle=\delta_{ab}
	\textup{ almost everywhere.}
	\]

    Since $(X,d,\cH^n)$ is a noncollapsed $\RCD(-K,n)$ space,
     its tangent module has constant dimension $n$; see
\cite[Theorem~1.12\textup{(vii)} and Corollary~2.14]
{De_Philippis_2018}.
Consequently, the $n$ pointwise orthonormal fields
$W^1,\ldots,W^n$ form an orthonormal basis of the tangent module
almost everywhere.
	
	
	Finally, on each $N_i$,
	\[
	|\diver W_i^a|
	\le \sqrt n\,|\nabla W_i^a|.
	\]
	Thus, for every $R>0$,
	\[
	\|\diver W_i^a\|_{L^2(B_R(q_i))}
	\longrightarrow0.
	\]
	The local version of divergence stability
	\cite[Lemma~2.15 and Remark~2.3]{Honda-Ketterer-Christian-Mondello-Perales-Rigoni} now gives
	\[
	W^a\in D_{\mathrm{loc}}(\diver),
	\qquad
	\diver W^a=0.
	\]
	This finishes the proof.
\end{proof}

\subsection{From $\RCD(-K,n)$ to $\RCD(0,n)$}
The pointed limit space obtained in \cref{prop:limit-frame} is a priori only an
$\RCD(-K,n)$ space.  In this subsection, we show that the existence of
global parallel, divergence-free orthonormal vector fields
$
W^1,\ldots,W^n$
improves it to an $\RCD(0,n)$ space.

Let $(X,d,\mathfrak m)$ be a metric measure space.  A function
$u\in W^{1,2}(X)$ belongs to the $L^2$-Laplacian domain $D(\Delta)$ if there
is a function, denoted $\Delta u\in L^2(X,\mathfrak m)$, such that
\[
\int_X\langle\nabla u,\nabla\phi\rangle\,d\mathfrak m
=-\int_X\phi\,\Delta u\,d\mathfrak m
\qquad \forall \phi\in W^{1,2}(X).
\]
On an $\RCD(K, N)$ space,  we use the following  test-function classes (\cite[Section 3.1]{Gigli_2018}\cite[Section 3.2]{Savar_2014})
\[
\begin{split}
	\TestF(X)&:=\{f\in D(\Delta)\cap L^\infty(X):
	|\nabla f|\in L^\infty(X),\ \Delta f\in W^{1,2}(X)\},\\
	\TestF_c(X)&:=\{f\in\TestF(X):\supp f\text{ is compact}\}.
\end{split}
\]
For $f\in\TestF(X)$, $\Hess f$ denotes its Hessian and
$|\Hess f|_{\mathrm{HS}}$ its pointwise Hilbert--Schmidt norm \cite[Section 3.1]{Gigli_2018}. 
Let
\[
\TestV(X)
:=
\left\{
\sum_{i=1}^k g_i\nabla f_i:
k\in\mathbb N,\quad f_i,g_i\in\TestF(X)
\right\}.
\]
Recall that $W^{1,2}(TX)$ is the space of vector fields
$V\in L^2(TX)$ admitting a weak covariant derivative
$
\nabla V\in L^2(T^{\otimes 2}X),$
equipped  with the norm
\[
\|V\|_{W^{1,2}(TX)}^2
=
\|V\|_{L^2(TX)}^2
+
\|\nabla V\|_{L^2(T^{\otimes 2}X)}^2.
\]
The space $H^{1,2}_C(TX)$ is defined by
\[
H^{1,2}_C(TX)
:=
\overline{\TestV(X)}^{\,W^{1,2}(TX)}.
\]
See \cite[Definitions~3.4.1 and~3.4.3]{Gigli_2018}.


%

\begin{proposition}\label{prop:RCD0n}
	Let $(X,d,\cH^n)$ be a complete noncollapsed $\RCD(K,n)$ space. Suppose that
	there exist vector fields
	\[
	W^1,\ldots,W^n
	\in
	W^{1,2}_{\loc}(TX)\cap L^\infty(TX)
	\cap D_{\loc}(\diver)
	\]
	which form an orthonormal basis of the local $L^0$-tangent module almost everywhere and satisfy
	\[
	\nabla W^a=0,
	\qquad
	\diver W^a=0,
	\qquad
	a=1,\ldots,n.
	\]
	Then $(X,d,\cH^n)$ is an $\RCD(0,n)$ space.
\end{proposition}

\begin{proof}
	We prove that $(X,d,\cH^n)$ satisfies the Bakry-\'Emery condition (or Bochner inequality)
	$\operatorname{BE}(0,n)$; see, for example, \cite[Definition 4.7]{Erbar_2014}. The conclusion then follows from the equivalence between
	$\operatorname{BE}(0,n)$ and $\RCD(0,n)$ 
	\cite{Erbar_2014}.
	
	Let $f\in\TestF_c(X)$. For $a=1,\ldots,n$, define 
	\[
	u_a:=\langle\nabla f,W^a\rangle.
	\]
	Since $\nabla f\in H^{1,2}_C(TX)$ and
	$W^a\in W^{1,2}_{\loc}(TX)$, it follows from 
	\cite[Theorem 3.4.2 (iv) \& Proposition~3.4.6 ]{Gigli_2018} that  $u_a \in W^{1, 2}_{\loc}(X)$ and 
	\[
	\nabla u_a
	=
	\Hess f(W^a,\cdot)^\sharp
	+
	\bigl\langle\nabla f,\nabla W^a\bigr\rangle^\sharp,
	\]
	where $\sharp$ is the musical isomorphism $\sharp\colon L^2(T^\ast X)\to L^2(TX)$. 
	
	Since $\nabla W^a=0$, it follows that
	\begin{equation}\label{eq:gradientofua}
		\nabla u_a=\Hess f(W^a,\cdot)^\sharp.
	\end{equation}
	Because $W^1,\ldots,W^n$ form a pointwise orthonormal basis of the tangent
	module, we have
	\begin{equation}\label{eq:gradientdecompose}
		\nabla f=\sum_{a=1}^n u_aW^a,
		\qquad
		|\nabla f|^2=\sum_{a=1}^n u_a^2.
	\end{equation}
	Using the fact that
	$\diver W^a=0$, we obtain
	\begin{align}
		\Delta f
		&=\diver(\nabla f)=\sum_{a=1}^n\diver(u_aW^a)\notag \\
		&=\sum_{a=1}^n
		\left(
		\langle\nabla u_a,W^a\rangle
		+
		u_a\diver W^a
		\right)\notag  \\
		&=\sum_{a=1}^n\langle\nabla u_a,W^a\rangle.  
		\label{eq:laplacian-trace}
	\end{align}
	
  Now let us prove  the  identity
	\begin{equation}\label{eq:laplace-ua}
		\Delta u_a=\langle\nabla\Delta f,W^a\rangle
	\end{equation}
	in the weak sense, that is,
	\[
	\int_X\langle\nabla u_a,\nabla\varphi\rangle\,d\cH^n
	=
	-\int_X\varphi\,
	\langle\nabla\Delta f,W^a\rangle\,d\cH^n,
	\] 
	for all $\varphi\in\TestF_c(X)$. Indeed, 
	by \eqref{eq:gradientofua}, we have 
	\[
	\int_X\langle\nabla u_a,\nabla\varphi\rangle\,d\cH^n
	=
	\int_X\Hess f(W^a,\nabla\varphi)\,d\cH^n.
	\]
	It follows from \cite[ Proposition~3.3.22]{Gigli_2018}  that
	\[
	W^a\bigl(\langle\nabla f,\nabla\varphi\rangle\bigr)
	=
	\Hess f(W^a,\nabla\varphi)
	+
	\Hess\varphi(W^a,\nabla f).
	\]
	Since $f$ and $\varphi$ are compactly supported and $\diver W^a=0$, we have 
	\[
	\int_X
	W^a\bigl(\langle\nabla f,\nabla\varphi\rangle\bigr)
	\,d\cH^n
	= -\int_X
	\bigl(\langle\nabla f,\nabla\varphi\rangle\bigr) \diver W^a
	\,d\cH^n = 0.
	\]
	Consequently,
	\[
	\int_X\Hess f(W^a,\nabla\varphi)\,d\cH^n
	=
	-\int_X\Hess\varphi(W^a,\nabla f)\,d\cH^n.
	\]
	We apply  \eqref{eq:gradientofua}
	with $f$ replaced by $\varphi$ and $u_a$ replaced by $\langle \nabla \varphi, W^a\rangle$, then 
	\[ \Hess\varphi(W^a,\nabla f) = \langle\nabla\langle\nabla\varphi,W^a\rangle,
	\nabla f\rangle.\]
	It follows that 
	\begin{align*}
		\int_X\Hess\varphi(W^a,\nabla f)\,d\cH^n
		&=
		\int_X\langle\nabla\langle\nabla\varphi,W^a\rangle,
		\nabla f\rangle\,d\cH^n\\
		&=
		-\int_X\langle\nabla\varphi,W^a\rangle\Delta f\,d\cH^n \quad \textup{ by definition of $\Delta$} \\
		&=
		\int_X\varphi\,
		\diver\bigl((\Delta f)W^a\bigr)\,d\cH^n \quad \textup{by definition of $\diver$}  \\
		&=
		\int_X\varphi\,
		\langle\nabla\Delta f,W^a\rangle\,d\cH^n \quad \textup{by the Leibniz rule of $\diver$ and $\diver W^a =0$.}
	\end{align*}
	This proves \eqref{eq:laplace-ua}.
	
	Applying the Laplacian to  \eqref{eq:gradientdecompose}, we obtain 
	\begin{align}
		\frac12\Delta|\nabla f|^2
		&=
		\frac12\sum_{a=1}^n\Delta(u_a^2)\notag\\
		&=
		\sum_{a=1}^n
		\left(
		u_a\Delta u_a+|\nabla u_a|^2
		\right).
		\label{eq:bochner-frame-first}
	\end{align}
	in the weak sense (see for example \cite[Equation (2.3.24)]{Gigli_2018}).
	On the other hand, we have 
	\begin{align}
		\langle\nabla f,\nabla\Delta f\rangle
		&=
		\sum_{a=1}^n
		u_a\langle W^a,\nabla\Delta f\rangle \textup{ by  \eqref{eq:gradientdecompose}}\notag \\
		&=
		\sum_{a=1}^n u_a\Delta u_a \textup{ by 
			\eqref{eq:laplace-ua}.}
		\label{eq:bochner-frame-second}
	\end{align}
	Subtracting \eqref{eq:bochner-frame-second} from
	\eqref{eq:bochner-frame-first}, we have 
	\begin{equation}\label{eq:exact-bochner-parallel-frame}
		\frac12\Delta|\nabla f|^2
		-
		\langle\nabla f,\nabla\Delta f\rangle
		=
		\sum_{a=1}^n|\nabla u_a|^2.
	\end{equation}
	Finally, \eqref{eq:laplacian-trace} and the Cauchy--Schwarz
	inequality imply
	\begin{align*}
		(\Delta f)^2
		&=
		\Big(
		\sum_{a=1}^n\langle\nabla u_a,W^a\rangle
		\Big)^2\leq
		n\sum_{a=1}^n
		\langle\nabla u_a,W^a\rangle^2 \leq
		n\sum_{a=1}^n|\nabla u_a|^2,
	\end{align*}
	where we used $|W^a|=1$ in the last inequality. This combined  with
	\eqref{eq:exact-bochner-parallel-frame} gives
	\[
	\frac12\Delta|\nabla f|^2
	-
	\langle\nabla f,\nabla\Delta f\rangle
	\geq
	\frac1n(\Delta f)^2
	\]
	in the weak sense, that is, we have 
	\[
	\frac12\int_X\Delta\varphi\,|\nabla f|^2\,d\cH^n
	-
	\int_X\varphi\,
	\langle\nabla f,\nabla\Delta f\rangle\,d\cH^n
	\geq
	\frac1n\int_X\varphi(\Delta f)^2\,d\cH^n
	\]
	for every nonnegative
	$\varphi\in \TestF_c(X)$.
	This is precisely the Bakry-\'Emery condition
	$\operatorname{BE}(0,n)$.
	Therefore $(X,d,\cH^n)$ is an $\RCD(0,n)$ space.
\end{proof}

\subsection{Euclidean volume growth and flatness}
\label{subsec:maximal-volume} In this subsection, we show that the uniformly controlled asymptotically flat ends give maximal
Euclidean volume growth on the pointed limit space.  As the limit space is an $\RCD(0,n)$ space, we use the sharp
Bishop--Gromov inequality to show that the limit space is isometric to the Euclidean space.

\begin{lemma}
	\label{lem:end-maximal-volume}
Fix an integer $n\ge3$, constants $b>0$, $\tau>(n-2)/2$ and $K\ge0$,
and a point
\[
p\in\R^n\setminus\overline B^{\mathbb E}_1(0).
\]
Let $(M_i,g_i,\Phi_i)$ be a sequence of connected, complete, one-ended,
$(b,\tau)$-asymptotically flat  $n$-manifolds  with $\Ric_{g_i}\ge-K$.
Let $p_i:=\Phi_i^{-1}(p)$, and suppose that $(M_i,d_{g_i},p_i,dV_{g_i})$ converges 
to $(X,d, p,\cH^n)$ in pointed measured Gromov-Hausdorff sense.
Then we have 
	\begin{equation}\label{eq:lowerbound}
		\limsup_{r\to\infty}
		\frac{\cH^n(B_r(p))}{\omega_n r^n}\ge1.
	\end{equation}
	where $
    \omega_n$
    is the volume of the Euclidean unit ball.
\end{lemma}

\begin{proof}
	Let us write $r=|x|$.  For  sufficiently large $r_0$, there exists a constant $C>0$ such that  
	\begin{equation}\label{eq:AF-uniform}
		|g_i-g_{\mathbb E}|\le Cr^{-\tau}
	\textup{ and }
		\sqrt{\det g_i}\ge1-Cr^{-\tau}.
	\end{equation}
	for all
	$r\ge r_0$, where $g_{\mathbb E}$ is the Euclidean metric on $\mathbb R^n$. 
	
    In $(\mathbb R^n, g_{\mathbb E})$, the annulus $A(r_0, R) = \{x\in \mathbb R^n: r_0<|x|< R\}$ is contained in $B_{C_1 + R-r_0}(p)$, 
    qwhere $C_1$ is the distance of $ p$ and the circle centered at the origin of radius $r_0$. It follows from \eqref{eq:AF-uniform} that $\Phi_i^{-1}(A(r_0, R))$ is contained in $B_{\rho(R)}(p_i)$, where 
    \[ \rho(R)=C_1+R-r_0+C\int_{r_0}^R s^{-\tau}ds.  \]
    Moreover, using \eqref{eq:AF-uniform}, we have  \begin{equation}\label{eq:AF-annulus-volume}
  	\vol_{g_i}(\Phi_i^{-1}(A(r_0,R)))
    		\geq \omega_n(R^n-r_0^n)
    		-C\int_{r_0}^{R}s^{n-1-\tau}\,ds = \omega_n R^n\bigl(1-\varepsilon(R)\bigr)
    \end{equation}
    for some $\varepsilon(R)$ such that $\varepsilon(R) \to 0$ as $R\to \infty$.
    
    It is clear that, under the assumptions of the lemma, the sequence 	$(M_i,d_{g_i},p_i,dV_{g_i})$  is volume noncollapsing. In particular, there exists $v>0$ such that 	$\vol_{g_i}(B_1(p_i))\ge v>0$ for all $i$.  
	
	It follows from 
	\cite[Theorem~1.3]{De_Philippis_2018} that 
	\begin{align*}
		\cH^n(B_{\rho(R)}(p))
		&=\lim_{i\to\infty}\vol_{g_i}(B_{\rho(R)}(p_i))\\
		&\ge\omega_n R^n\big(1-\varepsilon(R)\big).
	\end{align*}
	Observe that $\rho(R)/R\to 1$, as $R\to \infty$. Therefore, we have  
	\[ 	\limsup_{r\to\infty}
	\frac{\cH^n(B_r(p))}{\omega_n r^n}\geq 1.\]
\end{proof}

\begin{lemma}\label{lem:maximal-volume-flat}
	Let $(X,d,\cH^n)$ be a complete noncollapsed $\RCD(0,n)$ space.  If for
	some $p\in X$,
	\[
	\limsup_{r\to\infty}
	\frac{\cH^n(B_r(p))}{\omega_n r^n}\ge1,
	\]
	then $(X,d,p,\cH^n)$ is measure-preserving isometric to
	$(\R^n,d_{\mathbb E},0,\cH^n)$.
\end{lemma}

\begin{proof}
	For a given $p\in X$, we define
	\[
	\theta_p(r):=\frac{\cH^n(B_r(p))}{\omega_n r^n},
	\qquad r>0.
	\]

	Since $(X,d,\cH^n)$ is a  noncollapsed $\RCD(0,n)$ space, it follows from 
	\cite[Corollary~1.7]{De_Philippis_2018} that
	$\cH^n(B_r(p))\leq \omega_n r^n$ for all $r>0$, that is, $\theta_p(r)\leq 1$.  	According to the Bishop--Gromov comparison theorem,  $\theta_p$ is nonincreasing. This together with the assumption that \[
	\limsup_{r\to\infty}
	\frac{\cH^n(B_r(p))}{\omega_n r^n}\ge1,
	\] 
	implies 
	\begin{equation}\label{eq:all-radii-equality}
		\cH^n(B_r(p))=\omega_n r^n
		\qquad\text{for every }r>0.
	\end{equation}
	Furthermore, the equality case of \cite[Corollary~1.7]{De_Philippis_2018} states  that equality at radius $2R$
	implies the  ball $ B_R(p)$ is isometric to
	$ B_R^{\mathbb E}(0)$.  Letting $R$ go to infinity finishes the proof. 
\end{proof}

\subsection{Stability of the Positive Mass Theorem} 
We now combine the results of the preceding subsections to prove
\cref{thm:pmt-stability}.
\begin{proof}[Proof of \cref{thm:pmt-stability}]
	For each $i$, let $W_i^1,\ldots,W_i^n$ be the smooth vector fields from
	\cref{prop:almost-parallel-vector-fields}.  Then
	\begin{equation}\label{eq:L^2covariant}
		\sum_{a=1}^n\int_{M_i}|\nabla W_i^a|^2dV_{g_i}
		\le C(n) m_{\mathrm{ADM}}(M_i, g_i) \to 0 \textup{ as } i\to \infty. 
	\end{equation}
	Moreover, it follows from  \cref{lem:Gram-control} that 
	\begin{equation}\label{eq:L^2innerproduct}
		\sum_{a,b=1}^n\int_{B_R(p_i)}
		|g_i(V_i^a,V_i^b)-\delta_{ab}|^2\,dV_{g_i}\leq C_R \cdot m_{\mathrm{ADM}}(M_i, g_i) \to 0 \textup{ as } i\to \infty
	\end{equation}
	for every  $R>0$.
	
	 The Ricci lower bound and the Bishop--Gromov comparison theorem imply pointed
	Gromov--Hausdorff precompactness.  Hence, after passing to a
	subsequence, there exists a complete pointed metric space $(X,d,p)$
	such that $(M_i,d_{g_i},p_i)$ converges to  $(X,d,p)$ in the pointed Gromov-Hausdorff topology. 
Under the assumptions of the theorem, it is clear that the sequence 	$(M_i, d_{g_i}, p_i)$  is volume noncollapsing: there exists $v>0$ such that 
$\vol_{g_i}(B_1(p_i))\ge v$. It follows from  \cite[Theorem~1.2]{De_Philippis_2018}  that $(M_i,d_{g_i},dV_{g_i},p_i)$ converges $(X,d,\cH^n,p)$ in the pointed measured Gromov-Hausdorff topology,
where $(X,d,\cH^n)$ is a complete noncollapsed
$\RCD(-K,n)$ space.

	By  \cref{prop:limit-frame}, there are 
	bounded vector fields $W^1,\ldots,W^n$ on $(X,d,\cH^n)$ that form an orthonormal frame almost
	everywhere and satisfy $\nabla W^a=0$ and $\diver W^a=0$.  Hence
	\cref{prop:RCD0n} implies that  the limit $(X,d,\cH^n)$  is an
	$\RCD(0,n)$ space.  Now \cref{lem:end-maximal-volume} and  \cref{lem:maximal-volume-flat} show that   $(X,d,p,\cH^n)$ is measure-preserving isometric to 
	$(\R^n,d_{\mathbb E},0,\cH^n)$.
	
	Since the above argument holds for any  subsequence of $(M_i,d_{g_i},p_i,dV_{g_i})$,  it follows that every subsequence of $(M_i,d_{g_i},p_i,dV_{g_i})$ has a 
	sub-subsequence converging to the same pointed measured space. This proves that  $(M_i,d_{g_i},p_i)$ converges to the Euclidean space $(\R^n,d_{\mathbb E},0)$ in pointed Gromov-Hausdorff sense.
\end{proof}

\bibliography{bibliography}
\bibliographystyle{amsalpha}
\end{document}